\documentclass[10pt]{amsart}

\usepackage{amsmath,amssymb,amscd,amsxtra,upref}
\usepackage{mathrsfs}
\usepackage{latexsym}
\usepackage[dvips]{graphics,epsfig}
\usepackage{enumerate}
\usepackage[colorlinks=true, pdfstartview=FitV, linkcolor=blue, citecolor=blue, urlcolor=blue]{hyperref}
\usepackage{color}
\usepackage{bbm}
\usepackage{breqn}
\usepackage[left=3cm, right=3cm, top=3.5cm]{geometry}

 \makeatletter
\newtheorem*{rep@theorem}{\rep@title}
\newcommand{\newreptheorem}[2]{%
\newenvironment{rep#1}[1]{%
 \def\rep@title{#2 \ref*{##1}}%
 \begin{rep@theorem}}%
 {\end{rep@theorem}}}
\makeatother

\newtheorem{theorem}{Theorem}[section]
\newtheorem{lemma}[theorem]{Lemma}
\newtheorem{proposition}[theorem]{Proposition}
\newtheorem{corollary}[theorem]{Corollary}
\newreptheorem{proposition}{Proposition}
\newreptheorem{theorem}{Theorem}

\theoremstyle{remark}
\newtheorem{definition}{Definition}[section]

\newcommand{\R}{\mathbb{R}}
\newcommand{\N}{\mathbb{N}}
\newcommand{\Z}{\mathbb{Z}}
\newcommand{\Rn}{\mathbb{R}^n}
\newcommand{\Rnn}{\mathbb{R}^{2n}}
\newcommand{\C}{\mathbb{C}}

\newcommand{\He}{\mathbb{H}}
\newcommand{\Hn}{\mathbb{H}^{n}}
\newcommand{\rhol}{\rho_{\lambda}}

\newcommand{\Ma}{\mathcal{M}}

\newcommand{\dotp}{\boldsymbol{\cdot}}
\newcommand{\fhat}{\widehat{f}}
\newcommand{\muhat}{\widehat{\mu}}

\newcommand{\F}{\mathcal{F}}
\newcommand{\Sw}{\mathscr{S}}

\newcommand{\BL}{\mathcal{B}(L^2(\Rn))}
\newcommand{\UL}{\mathcal{U}(L^2(\Rn))}
\newcommand{\Sa}{\mathcal{S}}
\newcommand{\pla}{\phi_{\alpha,\lambda}}
\newcommand{\plb}{\phi_{\beta,\lambda}}
\newcommand{\SPL}{S_p(L^2(\Rn))}
\newcommand{\psihat}{\widehat{\psi}}
\newcommand{\Hhat}{\widehat{\He}^n}
\newcommand{\dhat}{d_{\Hhat}}
\newcommand{\Ht}{\widetilde{\He}^n}
\newcommand{\muhatH}{\widehat{\mu}_{\Hn}}
\newcommand{\fhatH}{\widehat{f}_{\Hn}}
\newcommand{\FH}{\mathcal{F}_{\Hn}}
\newcommand{\dht}{d_{\Ht}}

\newcommand{\psie}{\psi_{\epsilon}}

\newcommand{\mue}{\mu_{\epsilon}}
\newcommand{\muehat}{\widehat{\mu}_{\epsilon, \Hn}}

\begin{document}





\address{Fernando Rom\'{a}n Garc\'{i}a, Department of Mathematics, 250 Atgeld Hall, University of Illinois Urbana Champaign,
Urbana, Il 61801, USA.} \email{romanga2@illinois.edu}

\title[Fourier Coefficients and Hausdorf Dimension in the Heisenberg group]
{A Fourier Coefficients Approach to Hausdorff Dimension in the Heisenberg Group}
\author{Fernando Rom\'{a}n Garc\'{i}a}

\date{\today}

\begin{abstract}
This paper establishes connections between the group-Fourier transform and the geometry of measures in the Heisenberg group. Firstly, it is shown that if the Fourier transform of a compactly supported, finite, Radon measure is square integrable, then the measure must have a square integrable density. If it's Fourier transform is integrable, the the measure must have a continuous density. In addition, an alternative formulation of the Fourier transform on the Heisenberg group is used to show that energies of measures can be computed via integrals on an appropriate frequency space. This in turns opens the possibility of using Fourier methods in the computation of Hausdorff dimension of sets. 
\end{abstract}

\maketitle

\section{Introduction}
In studying the geometric properties of measures in $\Rn$, it is common to study their Fourier transforms. For $g\in L^1(\Rn)$,
\[\mathcal{F}_E (g)(\xi)=(2\pi)^{-n/2}\int e^{-ix\dotp \xi}g(x)dx.\]
Similarly, for a compactly supported, positive, and finite Radon measure, $\mu$, in $\Rn$, its distributional Fourier transform is given by  
\[\mathcal{F}_E(\mu)(\xi):=(2\pi)^{-n/2}\int_{\Rn}e^{- i x\cdot \xi}d\mu (x).\]
 As expected, $\mathcal{F}_E(\mu)$ and $\mu$ are closely related. For instance if $f$ is a function in the Schwartz class of $\Rn$, then
\begin{equation}\label{Eplanch}
\int_{\Rn}f(x)d\mu(x)=\int_{\Rn}\mathcal{F}_E(f)(x)\overline{\mathcal{F}_E(\mu)(x)}dx.
\end{equation}
These sort of relations between $\mathcal{F}_E(\mu)$ and $\mu$ are key in the application of Fourier theory to the potential-theoretic approach to Hausdorff dimension  introduced by R. Kauffman in \cite{Kaufman1}. This approach establishes a conection between Hausdorff dimension and energies of measures.
\begin{theorem}
Let $A\subset\Rn$ be a Suslin set, and denote by $\Ma(A)$ the collection of all positive and finite Radon measures with compact support in $A$. Then
\begin{equation}
\dim A=\sup\{\sigma>0:\exists\ \mu\in\Ma(A)\text{ s.t. }I_{\sigma}^E(\mu)<\infty\},
\end{equation}
where\begin{equation}
I_{\sigma}^E(\mu)=\int_{\Rn}\int_{\Rn}|x-y|^{-\sigma}d\mu(y)d\mu(x).
\end{equation}
\end{theorem}
 This results holds more generally for Suslin subsets of complete metric spaces, by replacing the integrand, $|x-y|^{-\sigma}$, with the distance between $x$ and $y$, $d(x,y)^{-\sigma}$. In Euclidean space, \eqref{Eplanch} allows the computation of energies of measures via integrals on the frequency domain. More precisely, there is a constant $c (n,\sigma)$, depending only on $n$ and $\sigma$, such that
\begin{equation}\label{FEnergy}
I_{\sigma}^E(\mu)=c(n,\sigma)\int_{\Rn}|\xi|^{\sigma-n}|\mathcal{F}(\mu)(\xi)|^2d\xi.
\end{equation}
 A proof of this result can be found in \cite[pp. 38]{Mattila1}. Another result connecting properties of $\mathcal{F}_E(\mu)$ to $\mu$ itself is the following,
\begin{theorem} \label{Th2}
Let $\mu$ be a compactly supported, positive, finite, Radon measure on $\Rn$.
\begin{enumerate}
\item If $\mathcal{F}_E(\mu)\in L^2(\Rn)$, then $\mu = gdx$ with $g\in L^2(\Rn)$.
\item If $\mathcal{F}_E(\mu)\in L^1(\Rn)$, then $\mu = hdx$ with $h\in\mathcal{C}(\Rn).$
\end{enumerate}
\end{theorem}
 This result has been successfully used in answering several important questions in geometric measure theory in $\Rn$. Marstrand-Mattila projection theorem (\cite{Marstrand}, \cite{Mattila3}), dimension of intersections of general sets (\cite{Mattila6},\cite{Falconer4}), and the distance set problem (\cite{Erdogan}), are a few example of problems where Theorem \ref{Th2} has played a crucial role.
A proof of this theorem can be found in \cite[pp. 31]{Mattila1}. 

In this work the aim is to establish a connections between the Fourier transform and the geometry of measures in the Heisenberg group.%

The main difficulty of Fourier theory in $\Hn$ is that the Fourier transform of an integrable function is not a complex valued function but an operator-valued map from $\R\setminus\{0\}$ to $\BL$, the space of bounded operators on $L^2(\Rn)$. Nevertheless, it is possible to obtain analogues of many of the classical results. For instance an appropriate analogues of Plancherel's theorem and the inversion formula for the Schwartz class hold. 

Another important result is the analogue of  \eqref{Eplanch},
\begin{theorem}\label{prop1}
If $f$ is in the Schwartz class of $\Hn$, $\mu$ is a compactly supported, positive, finite, Radon measure on $\Hn$, and $d\varrho(\lambda)=\frac{|\lambda|^n}{(2\pi)^{n+1}}d\lambda$, then
\begin{equation}
\int_{\Hn}f(p)d\mu(p)=\int_{\R^*}tr(\muhat(\lambda)^*\fhat(\lambda))d\varrho(\lambda).
\end{equation}
\end{theorem}
Here, $tr(T)$ denotes the trace of the operator $T$, and $T^*$ denotes its adjoint. 
Exploiting some of these properties we obtain the following analogue of Theorem \ref{Th2}.
\begin{proposition}\label{prop2}
Suppose $\mu$ is a compactly supported, positive, finite, Radon measure on $\Hn$. Then,
\begin{enumerate} [(i)]
\item If $\muhat\in L^2(\Sa_2)$, then $d\mu=f dp$ with $f \in L^2(\Hn)$.
\item If $\muhat\in L^1(\Sa_1)$, then $d\mu=g dp$ with $g \in \mathcal{C}(\Hn)$,
\end{enumerate}
where $L^p(\Sa_p)$ is an appropriately defined, non-commutative, $L^p$-space. 
\end{proposition}
However, one of the consequences of the operator nature of the Fourier transform is that it makes it quite difficult to characterize the image of certain classes of functions. For instance, one would like to know the image of the Schwartz class, $\Sw(\Hn)$, under the group Fourier transform. This was achieved by D. Geller in \cite{Geller}, where he obtained a characterization in terms of asymptotic series. It is not simple, however, to use this characterization to extend the Fourier transform to distributions. Recently in \cite{BCD1}, the authors proposed an alternative frequency space for the Heisenberg group, and with it, an alternative definition of the Fourier transform whereby the Fourier transform of an integrable function is now a continuous function defined on this frequency space. Later,  in \cite{BCD2}, the authors were able characterize the image of $\Sw(\Hn)$ under this new ``Fourier transform", and extend the definition to tempered distributions. In this work, this approach is used to obtain a Heisenberg group analogue of \eqref{FEnergy}, where now 
\[I_{\sigma}(\mu):=\int_{\Hn}\int_{\Hn}||q^{-1}p||_{\Hn}^{-\sigma}d\mu(q)d\mu(p).\]
 The result reads,
\begin{proposition}\label{FHenergy}
Let $\mu\in\Ma(\Hn)$, then there is a constant $d(n,\sigma)$ depending only on $n$ and $\sigma$ such that,
\begin{equation}
I_{\sigma}(\mu)=d(n,\sigma)\int_{\widehat{\He}^n}\mathfrak{K}_{Q-\sigma}(\alpha,\lambda)|\widehat{\mu}_{\Hn}(\zeta)|^{2}d\zeta.
\end{equation}
\end{proposition}
Here, $Q=2n+2$ is the Hausdorff dimension of $\Hn$, $\widehat{\He}^n$ denotes the aforementioned frequency space, $d\zeta$ is an appropriately defined measure, and $\mathfrak{K}_{Q-\sigma}$ is related to the distributional Fourier transform of $K_s(p)=||p||_{\Hn}^{-s}$ in the sense of \cite{BCD1}.

\section{Group Fourier transform in the Heisenberg group}
The Heisenberg group is the manifold $\Hn=\R^{2n}\times\R=\R^n\times\R^n\times\R$ with the group law $(z, t)(z', t')=(z+z', t+t'-\frac{1}{2}\omega(z,z'))$, where $\omega$ denotes the standard symplectic form. That is, if $z=(x,y), z'=(x', y')$ then $\omega(z, z')=x\cdot y'-y\cdot x'$. This group law turns $\Hn$ into a locally compact, non-Abelian Lie Group with respect to the usual topology in $\R^{2n+1}$. The left invariant vector fields are 
\[X_j:=\frac{\partial}{\partial x_j}-\frac{y_j}{2}\frac{\partial}{\partial t}; Y_j:=\frac{\partial}{\partial y_j}+\frac{x_j}{2}\frac{\partial}{\partial t}, \text{ for } j=1,\ldots,n,\]
\[\text{and}, T:=\frac{\partial}{\partial t}.\]
The center of $\Hn$ is the t axis, $\{(0,t):t\in\R\}$. One can check that up to multiplication by a non-zero constant, the Haar measure on $\Hn$ is the $(2n+1)$-dimensional Lebesgue measure, therefore, for $p\in [0,\infty]$, the space $L^p(\Hn)$ coincides with the space $L^p(\R^{2n+1})$. Moreover, since the differential structure of $\Hn$ is that of $\R^{2n+1}$ differentiable functions on $\Hn$ are those that are differentiable in the $\R^{2n+1}$ sense. In particular, the Schwartz class in $\Hn$, here denoted by $\mathscr{S}(\Hn)$, is simply $\mathscr{S}(\R^{2n+1})$. The Koranyi gauge, $||(z,t)||_{\Hn}^4=|z|^4+16t^2$, induces a left invariant metric, known as the Koranyi metric, given by $d_{\Hn}(p,q)=||q^{-1}p||_{\Hn}$. For any $r>0$, the map $\delta_r(z,t)=(rz,r^2t)$ is a group automorphism and is homogeneous of degree 1 with respect to the left invariant metric $d_{\Hn}$.  Much more can be said about the structure of the Heisenberg group but much of it will not be of great relevance to us. For a detailed introduction to the Heisenberg group, readers are referred to \cite{CDPT}.\\

As for the general theory of Fourier analysis on groups, in order to introduce the Fourier transform, one first needs a complete description of the representations of the group. The representation theory of $\Hn$ is relatively simple. Denote by $\UL$ the group of unitary operators on $L^2(\R^n)$. For each $\lambda\in\R^*:=\R\setminus\{0\}$, the map
\[\rhol:\Hn\to \mathcal{U}(L^2(\R^n))\]
given by \[\rhol(x,y,t)\varphi(\xi)=e^{i\lambda (t+x\cdot\xi+\frac{xy}{2})}\varphi(\xi+y),\quad (\varphi\in L^2(\R^n)),\] 
 is a strongly continuous, unitary representation of the Heisenberg group. Moreover, one can check that these representations are irreducible, and  by Stone-von Neumann Theorem (\cite{Stone}, \cite{VonN}), up to unitary equivalence, these are all the representations of $\Hn$ that are non-trivial at the center. Details about this can be found in many places, see for instance \cite[Section 1.2]{Tangavelu1}, or \cite[Section 1.5]{Folland2}.

This allows us to introduce the Fourier transform of integrable functions in $\Hn$. Denote by $\BL$ the space of bounded operators on $L^2(\Rn)$, then the Fourier transform of $f\in L^1(\Hn)$ is the operator-valued function 
\[\F(f)(\lambda)=\fhat(\lambda):\R^*\to\BL,\]
given by
\begin{equation}
\fhat(\lambda)\varphi=\int_{\Hn}f(p)\rhol(p)\varphi dp.
\end{equation}
The integral is in the Bochner sense, that is to say $\fhat(\lambda)\varphi$ is the function such that for every $\psi\in L^2(\Rn)$,
\[\langle \fhat(\lambda)\varphi,\psi\rangle_{L^2(\Rn)}=\int_{\Hn}f(p)\langle\rhol(p)\varphi , \psi \rangle_{L^2(\Rn)}dp.\]

It is not hard to see that for each $\lambda$, $\fhat(\lambda)$ is indeed bounded. In fact, if we denote by $||\cdot||_{op}$ the operator norm in $\BL$, we have $||\fhat(\lambda)||_{op}\leq ||f||_1$. As expected, the Fourier transform takes convolutions to ``products", where now the product is composition of operators. Since neither the group law, nor composition of operators are commutative one has to be consistent with the order in which the functions are placed in the convolution. Here,
 \[f*g(p)=\int_{\Hn}f(pq^{-1})g(q)dq,\]
so that 
\[\widehat{f*g}(\lambda)=\fhat(\lambda)\widehat{g}(\lambda).\]
As mentioned before, the group Fourier transform shares many properties with the classical Fourier transform in $\Rn$. In what follows some of this properties will be discussed, many of them follow from the more general theory of Fourier theory on groups but for the sake of completeness Heisenberg-specific proofs are included.\\

 As one expects from the Fourier transform, depending on the regularity of $f$, one can obtain more than just boundedness of the operators $\fhat(\lambda)$. To state this formally, I first need to introduce more notation.

A compact operator $T$ is said to be in the Schatten p-class, denoted $S_p(L^2(\Rn))$, if its singular values are in $\ell_p$. That is
\[tr[(TT^*)^{p/2}]<\infty.\]
The Schatten p-norm of $T\in \SPL$ is $||T||_{S_p}=Tr[(TT^*)^{p/2}]^{1/p}$. Note that $S_1$ is the trace class and $S_2$ is the class of Hilbert-Schmidt operators. 

Now, let $\varrho$ be the measure on $\R^*$ given by
\[d\varrho(\lambda)=\frac{|\lambda|^n}{(2\pi)^{n+1}},\]
and denote by $L^p(\R^*,\SPL; d\varrho)$ the space of $\SPL$-valued functions on $\R^*$ which are $\varrho$ measurable and $\int_{\R*}||T(\lambda)||_{S_p}^pd\varrho(\lambda)<\infty$.

The spaces $L^p(\R^*,\SPL; d\varrho)$ are examples of non-commutative $L^p$ spaces. In particular, the space $L^2(\R^*,S_2(L^2(\Rn)); d\varrho)$ is a non-commutative Hilbert space with inner product
\[\langle T, S \rangle_{L^2(S_2)}=\int_{\R^*}tr(T(\lambda)S(\lambda)^*)d\varrho(\lambda).\]

The simplified notation $L^p(S_p)$ is used to denote these spaces. With these definitions in place, we can continue our discussion by drawing parallels between the Fourier transform in $\Hn$ and the classical Euclidean theory. For instance, one very important property of the classical Fourier transform is Plancherel's theorem. Turns out, there is a non-commutative analogue that holds in $\Hn$.

\begin{theorem}[Plancherel Theorem]\label{Plancherel}
If $f\in L^1(\Hn)\cap L^2(\Hn)$ then $\fhat\in L^2(\Sa_2)$ with 
\begin{equation}\label{Parseval}
||f||_2=||\fhat||_{L^2(\Sa_2)}.
\end{equation}
Moreover, if $f$ and $g$ are in $L^1(\Hn)\cap L^2(\Hn)$ then 
\begin{equation}\label{polarized}
\int_{\Hn}f(p)g(p)dp=\int_{\R^*}tr(\widehat{g}(\lambda)^*\fhat(\lambda)d\varrho(\lambda)
\end{equation}
\end{theorem}

\begin{proof}
The first thing to note is that $\fhat(\lambda)$ is an integral operator in $L^2(\Rn)$, with kernel given by
\[T_f^{\lambda}(\eta,\xi)=\int_{\Rn}f^{\lambda}(x,\eta-\xi)e^{i\lambda x\dotp(\frac{\xi+\eta}{2})}dx,\]
where \[f^{\lambda}(z)=\int_{\R}f(z,t)e^{i\lambda t}dt.\] Therefore, $||\fhat(\lambda)||_{S_2}^2=||T_f^{\lambda}||^2_{L^2(\Rn\times\Rn)}.$ For $g\in L^2(\R^{2n})$, denote by $\mathcal{F}_E^jg$ the Euclidean Fourier transform of $g$ in the first j coordinates. Applying change of variables and Euclidean Plancherel theorem, we have

\begin{equation}\label{eq1}
\begin{split}
||T_f^{\lambda}||^2_{L^2(\R^{2n})}&=\int_{\Rn}\int_{\Rn}(2\pi)^{n/2}|\mathcal{F}_E^nf^{\lambda} \left(\frac{\lambda(\eta+\xi)}{2},\eta-\xi\right)|^2d\eta d\xi\\
&=\left(\frac{2\pi}{|\lambda|}\right)^n\int_{\R}\int_{\R}|f^{\lambda}(\eta, \xi)|^2d\eta d\xi
\end{split}
\end{equation}

Recalling the definition of $f^{\lambda}$, it follows by Plancherel theorem in $\R$ that \[\int|f^{\lambda}(z)|^2d\lambda=2\pi\int|f(z,t)|^2dt.\]
Therefore, the claim follows by integrating both sides of (\ref{eq1}) with respect to $\varrho$.
Equation (\ref{polarized}) follows by polarizing (\ref{Parseval}).

\end{proof}

Since $L^1(\Hn)\cap L^2(\Hn$) is dense in $L^2(\Hn)$, Plancharel theorem extends to all of $L^2(\Hn)$ making the Fourier transform a Hilbert space isomorphism between $L^2(\Hn)$ and $L^2(\Sa_2)$. 

Another important aspect of the Euclidean Fourier transform is its interaction with the Schwartz class. The space $\Sw(\Hn)$ coincides with $\Sw(\R^{2n+1})$. As expected, the regularity of functions in $\Sw(\Hn)$ translate as regularity of their Fourier transform (this will be made much more precise later) and allows the following inversion formula analogous to the Euclidean one. 

\begin{theorem}\label{SwTrace}

\begin{enumerate}[(i)]\text{ }

\item If $f\in\Sw(\Hn)$ then $\fhat\in L^p(\Sa_p)$, for $1\leq p\leq \infty$.
\item If $f\in\Sw(\Hn)$, then 
\begin{equation}\label{inversion}
f(p)=\int_{\R^*}tr(\rhol(p)^*\fhat(\lambda))d\varrho(\lambda).
\end{equation}

\end{enumerate}

\end{theorem}

\begin{proof}
\begin{enumerate}
\item 
Fist, it is clear that $f\in L^2(\Hn)$, so that $\fhat(\lambda)\in S_2$. In particular, $\fhat(\lambda)$ is compact. Since for all $1<p\leq\infty$, $S_p\subset S_1$, it is enough to show $\fhat\in L^1(\Sa_1)$. That is, $\int_{\R^*}tr(|\fhat(\lambda)|^{1/2})d\varrho(\lambda)<\infty$. 
If $\{e_j\}_{j\in\mathbb{N}}$ is an orthonormal basis of $L^2(\Rn)$, and $s_j(A)$ denotes the $j^{th}$ singular value of $A$. The singular values satisfy
\begin{equation}\label{snumbers}
\sum_{j=1}^{\infty}s_j(A)\leq\sum_{i,j=1}^{\infty}|\langle Ae_i , e_j \rangle_{L^2(\Rn)}|.
\end{equation}
 Let $\{\pla\}_{\alpha\in\mathbb{N}^n}$ be the orthonormal basis of $L^2(\Rn)$ comprised of the re-scaled multidimensional Hermite functions (discussed in detail at the beginning of Section \ref{FCoeffH}). For $\fhat(\lambda)$, equation (\ref{snumbers}) re-writes as
 \begin{equation}\label{boundonsnumbers}
 \sum_{\alpha\in\mathbb{N}^n}s_{\alpha}(\fhat(\lambda))\leq \sum_{\alpha,\beta}|\langle \fhat(\lambda)\pla , \plb \rangle_{L^2(\Rn)}|.
 \end{equation}
 Since $||\fhat(\lambda)||_{S_p}=\sum_{\alpha\in\mathbb{N}^n}s_{\alpha}(\fhat(\lambda))$, it is enough to show that the right hand side of (\ref{boundonsnumbers}) is integrable with respect to $\varrho$. This follows from Lemma 2.1 in \cite{BCD2} which reads,
 \begin{lemma}\label{decayF}
 For any $q\in\mathbb{N}$, there exist $N_q\in\mathbb{N}$ and constant $C_q>0$ such that
 \begin{equation}\label{decay}
 \left[ 1+|\lambda|(|\alpha|+|\beta|+n)+|\alpha-\beta| \right]^q|\langle \fhat(\lambda)\pla,\plb \rangle_{L^2(\Rn)}|\leq C_q||f||_{N_q,\Sw},
 \end{equation}
 where $||f||_{N_q,\Sw}$ is the classical family of Schwartz semi-norms,
 \[||f||_{N,\Sw(\Hn)}:=\sup_{|\alpha|\leq N}||(1+|z|^2+t^2)^{N/2}\partial^{\alpha}_{z,t}f||_{L^{\infty}(\Hn)}.\]
 \end{lemma}
 As mentioned before, it follows directly that
 \begin{align*}
 ||\fhat||_{L^1(\Sa_1)}&=\int_{\R^*}tr(|\fhat(\lambda)|)d\varrho(\lambda)\\
 &\leq\int_{\R^*} \sum_{\alpha,\beta}|\langle \fhat(\lambda)\pla,\plb \rangle_{L^2(\Rn)}| d\varrho(\lambda)<\infty
 \end{align*}
 
 \item 
To prove the inversion formula first note that taking $p=1$ in part (i), we get that (\ref{inversion}) is well defined. Since $\fhat(\lambda)$ is an integral operator, it is clear that so is $\rhol(p)^*\fhat(\lambda)$. The kernel of  $\rhol(p)^*\fhat(\lambda)$ is \[K_f^{\lambda,p}=e^{-i\lambda t}e^{i\lambda(\frac{x\dotp y}{2}-x\dotp\xi)}T_f^{\lambda}(\eta, \xi-y),\]
where $p=(x,y,t)$. The formula for the trace of integral operators (i.e $tr(\rhol(p)^*\fhat(\lambda))=\int_{\Rn}K_f^{\lambda,p}(\eta, \eta)d\eta$ - see \cite{Brislawn}), gives
\begin{equation}
\begin{split}
&tr(\rhol(p)^*\fhat(\lambda))\\&=e^{-i\lambda t}\int_{\Rn}e^{-i\lambda x\dotp( \eta+\frac{y}{2})}\int_{\Rn}f^{\lambda}(u,y)e^{i\lambda u\dotp (\eta-\frac{y}{2})}dud\eta\\
&=\frac{(2\pi)^n}{|\lambda|^n}e^{-i\lambda t}f^{\lambda}(x,y).
\end{split}
\end{equation}
Here, the last equality follows by using frequency modulation and Fourier inversion formula in Euclidean space, together with the change of variables $\eta\to\frac{\eta}{\lambda}$. Integrating both sides with respect to $d\varrho(\lambda)$ and using Euclidean Fourier inversion in the $\lambda$ variable, completes the proof.
 \end{enumerate}
\end{proof}

In analogy with Euclidean space, if $\mu\in\Ma(\Hn)$ one defines its Fourier transform as the $\BL$-valued map $\muhat:\R^*\to\BL$ given by
\begin{equation}\label{FofMeas}
\muhat(\lambda)\varphi=\int_{\Hn}\rhol(p)\varphi d\mu(p).
\end{equation}
This is a specific case of the more general theory of characteristic functions of measures on locally compact groups (see for instance H. Heyer \cite{Heyer} and E. Siebert \cite{Siebert}). It is not hard to see that $\muhat(\lambda)$ is indeed bounded with $||\muhat(\lambda)||_{op}\leq\mu(\Hn).$ Many properties of the Fourier transform of functions also extend to measures. For instance, group convolutions are defined as usual by
\[f\ast\mu(p)=\int_{\Hn}f(pq^{-1})d\mu(q),\]
and their Fourier transform is a composition of operators
\[\widehat{f\ast\mu}=\fhat(\lambda)\muhat(\lambda).\]

In studying the group Fourier transform of measures, it is extremely useful to use convolution approximations. Approximations to the identity in the convolution algebra $L^1(\Hn)$ coincide with classical approximations to the identity in $L^1(\Rn)$. The following properties are easy to check

\begin{lemma}\label{PropOfAI}
Let $\{\psi_{\epsilon}\}_{\epsilon>0}$ be an approximation to the identity in $L^1(\Hn)$.
\begin{enumerate}
\item $\psi_{\epsilon}\to\delta_{0}$ in the weak sense.
\item If $f\in L^1(\Hn)$, $\psi_{\epsilon}\ast f\to f$ in $L^1(\Hn)$.
\item If $g\in\mathcal{C}(\Hn)$ is bounded,  $\psi_{\epsilon}\ast g\to g$ point-wise.
\item If $\mu\in\Ma(\Hn),\ \psi_{\epsilon}\ast\mu\to\mu$ in the weak sense. 
\end{enumerate}
\end{lemma}
Here, $\delta_0$ denotes the Dirac distribution. It is possible, by way of \eqref{FofMeas}, to explicitly compute $\widehat{\delta}_0$. For $\varphi\in L^2(\Rn)$,
\[\widehat{\delta}_0(\lambda)\varphi=\int_{\Hn}\rhol(p)\varphi d\delta_0(p)=\rhol(0)\varphi=\varphi.\]
That is, $\widehat{\delta}_{0}=I$, the identity operator. One expects that the weak convergence of approximations to the identity, $\psi_{\epsilon}$,  implies some form of convergence of $\psihat_{\epsilon}$ to $I$. This is indeed the case

\begin{lemma}
Let $\{\psi_{\epsilon}\}$ be an approximation to the identity. Then for each $\lambda$,
 $\widehat{\psi}_{\epsilon}(\lambda)\to I$, in the strong operator topology.
 \end{lemma}
 Note that a corollary of this proposition, and the convolution theorem for measures, is that if $\mu\in\Ma(\Hn)$, the function $\mu_{\epsilon}=\psi_{\epsilon}*\mu$ satisfies $\widehat{\mu_{\epsilon}}(\lambda)\to \muhat(\lambda)$ in the strong operator topology.
 \begin{proof}
 Fix $\lambda\in\R^*$ and $\varphi\in L^2(\Rn)$. The aim is to show that 
 \[||(\widehat{\psi}_{\epsilon}(\lambda)-I)\varphi||_{L^2(\Rn)}\to 0 \text{ as } \epsilon\to0.\]
 Computing $||(\widehat{\psi}_{\epsilon}(\lambda)-I)\varphi||_{L^2(\Rn)}^2$ yields,
 \begin{equation}\label{SOTnorm}
 \begin{split}
&||(\widehat{\psi}_{\epsilon}(\lambda)-I)\varphi||_{L^2(\Rn)}^2 =\langle(\widehat{\psi}_{\epsilon}(\lambda)-I)\varphi, (\widehat{\psi}_{\epsilon}(\lambda)-I)\varphi \rangle_{L^2(\Rn)}\\
&= \langle \widehat{\psi}_{\epsilon}(\lambda)\varphi, \widehat{\psi}_{\epsilon}(\lambda)\varphi \rangle_{L^2(\Rn)} - \langle\varphi, \widehat{\psi}_{\epsilon}(\lambda)\varphi\rangle_{L^2(\Rn)} - \langle \widehat{\psi}_{\epsilon}(\lambda)\varphi, \varphi \rangle_{L^2(\Rn)} +\langle \varphi, \varphi \rangle_{L^2(\Rn)}\\
&=\int_{\Hn}\psi_{\epsilon}(p)\left(\int_{\Hn}\psi_{\epsilon}(q)\langle \rhol(p)\varphi, \rhol(q)\varphi \rangle_{L^2(\Rn)} dq \right)dp\\
&-\int_{\Hn}\psi_{\epsilon}(p)\langle \rhol(p)\varphi, \varphi \rangle_{L^2(\Rn)} dp-\int_{\He}\psi_{\epsilon}(p)\langle \varphi, \rhol(p)\varphi\rangle_{L^2(\Rn)} dp+\langle \varphi, \varphi\rangle_{L^2(\Rn)}.
\end{split}
 \end{equation}
Since $\rhol$ is a strongly continuous representation, all the integrands are continuous and bounded therefore, using the fact that $\psi_{\epsilon}\to\delta_0$ weakly, one sees that as $\epsilon\to 0$, the right hand side of \eqref{SOTnorm} converges to $\langle \varphi, \varphi \rangle_{L^2(\Rn)}-\langle \varphi, \varphi \rangle_{L^2(\Rn)}-\langle \varphi, \varphi \rangle_{L^2(\Rn)}+\langle \varphi, \varphi \rangle_{L^2(\Rn)}=0$. This completes the proof.
\end{proof}
As mentioned in the introduction, a quick consequence of the extension of $\mathcal{F}_{E}$ to distributions is that if $f\in\Sw(\Rn)$ and $\mu\in\Ma(\Rn)$ then
\[\int_{\Rn}fd\mu=\int_{\Rn}\mathcal{F}_E(f)\overline{\mathcal{F}_E(\mu)}.\]
Since, as of now, there is no satisfactory extension of the group Fourier transform to the space of distributions, the analogous result requires additional work. The following lemma will be useful,

\begin{lemma}\label{lemmaB}
 Let $H$ be a Hilbert space and for each $\epsilon>0$, $T_{\epsilon},T\in S_1( H)$. Suppose that $T_{\epsilon}\to T$ in the strong operator topology, then there is a subsequence $\{T_k\}_{k\in\mathbb{N}}\subset \{T_{\epsilon}\}_{\epsilon>0}$ such that $tr(T_k)\to tr(T)$ as $k\to\infty$.
 \end{lemma}

 \begin{proof}
 Let $\{e_j\}_{j\in\mathbb{N}}$ be a fixed orthonormal basis of $H$. Since $T_{\epsilon}\to T$ in the strong operator topology, it follows that for each $j$, $|\langle (T_{\epsilon}-T)e_j,e_j\rangle_{H}|\to0$ as $\epsilon\to0$. Therefore, we can obtain a subsequence $\{T_k\}_{k\in\mathbb{N}}$ such that for all $j,\  |\langle (T_k-T)e_j,e_j \rangle_{H}|\to0$ monotonically. Moreover, for each $k$, both $T_k$ and $T$ are traceable so $T_k-T$ is also traceable. In particular, by  Weyl's additive inequality,
 \[\sum_{j=1}^{\infty} |\langle (T_1-T)e_j,e_j \rangle_{H}|\leq tr(|T_1-T|)<\infty.\]
Now,
 \begin{align*}
 |tr(T_k)-tr(T)|&=|tr(T_k-T)|\\
&= |\sum_{\j=1}^{\infty}\langle (T_k-T)e_j,e_j \rangle_{H}|\\
& \leq \sum_{j=1}^{\infty} |\langle (T_k-T)e_j,e_j \rangle_{H}|.
 \end{align*}
 So the proof is completed by appealing to the decreasing monotone convergence theorem.
 \end{proof}

\begin{theorem}\label{Fmu}
Let $f\in\Sw(\Hn)$ and $\mu\in\Ma(\Hn)$, then
\begin{equation}
\int_{\Hn}f(p)d\mu(p)=\int_{\R^*}tr(\fhat(\lambda)\muhat(\lambda)^*)d\varrho(\lambda)
\end{equation}
\end{theorem} 
\begin{proof}
Let $\{\psi_{\epsilon}\}_{\epsilon>0}\in\mathcal{C}^{\infty}_c(\Hn)$ be a compactly supported, non-negative, smooth approximation to the identity, and put $\mu_{\epsilon}=\psi_{\epsilon}*\mu$. It is clear that $\mu_{\epsilon}$ is a compactly supported smooth function, in particular $\mu_{\epsilon}\in\Sw(\Hn)$. Also, $\mu_{\epsilon}\to\mu$ weakly, and for each $\lambda$, $\widehat{\mu}_{\epsilon}(\lambda)\to\muhat(\lambda)$ in the strong operator topology. Therefore, $\fhat(\lambda)\widehat{\mu}_{\epsilon}(\lambda)^*\to\fhat(\lambda)\muhat(\lambda)^*$ in the strong operator topology.
 Since $f\in\Sw(\Hn)$, then for each $\lambda\in\R^*$, $\fhat(\lambda)\in S_1(L^2(\Rn))$. It follows that for each $\epsilon>0, \lambda\in\R^*$, $\fhat(\lambda)\widehat{\mu}_{\epsilon}(\lambda)^*$, and $\fhat(\lambda)\muhat(\lambda)^*$ are in $S_1(L^2(\Rn))$.
By Lemma \ref{lemmaB}, there is a subsequence $\{\mu_k\}_{k\in\mathbb{N}}\subset\{\mu_{\epsilon}\}_{\epsilon>0}$ such that
\begin{enumerate}[(i)]
\item $\mu_k\to\mu$ weakly,
\item $\muhat_k(\lambda)\to\muhat(\lambda)$ in the strong operator topology, and
\item $tr(\fhat(\lambda)\muhat_k(\lambda)^*)\to tr(\fhat(\lambda)\muhat(\lambda)^*)$.
\end{enumerate}  
By Theorem \ref{Plancherel} we have,
\begin{equation}\label{eq3}
\begin{split}
\int_{\Hn}f(p)d\mu(p)&=\lim_{k\to\infty}\int_{\Hn}f(p)\mu_k(p)dp\\ 
&=\lim_{k\to\infty}\int_{\R^*}tr(\fhat(\lambda)\widehat{\mu}_k(\lambda)^*)d \varrho(\lambda)
\end{split}
\end{equation}
Now, by the additive Weyl's inequality, we have 
\begin{align*}
|tr(\fhat(\lambda)\widehat{\mu}_k(\lambda)^*)|&\leq tr(|\fhat(\lambda)\widehat{\mu}_k(\lambda)^*|)\\
& \leq ||\widehat{\mu}_k||_{op}||\fhat(\lambda)||_{S_1}\\
&\leq ||\widehat{\psi}_k(\lambda)||_{op}||\muhat(\lambda)||_{op}||\fhat(\lambda)||_{S_1}\\
&\leq \mu(\Hn)||\fhat(\lambda)||_{S_1}.
\end{align*}
By Theorem \ref{SwTrace}, since $f\in\Sw(\Hn)$, $\fhat\in L^1(\Sa_1)$, so by the dominated convergence theorem, the limit can be brought inside the integral on the right hand side of (\ref{eq3}) to complete the proof.

\end{proof}

With this at hand, we are ready to prove Proposition \ref{prop2},

\begin{repproposition}{prop2}
Let $\mu\in\Ma(\Hn)$.
\begin{enumerate} [(i)]
\item If $\muhat\in L^2(\Sa_2)$, then $d\mu=f dp$ for $f \in L^2(\Hn)$.
\item If $\muhat\in L^1(\Sa_1)$, then $d\mu=g dp$ for $g \in \mathcal{C}(\Hn)$.
\end{enumerate}
\end{repproposition}

\begin{proof}
\begin{enumerate}[(i)]
\item Since the Fourier transform is an isometric isomorphism between $L^2(\Hn)$ and $L^2(\Sa_2)$, it follows that $\exists\ f\in L^2(\Hn)$ such that $\fhat(\lambda)=\muhat(\lambda)$ for $ \varrho-a.e.\ \lambda$. The aim is to show that this $f$ is the density of $\mu$. As before, let $\{\psi_{\epsilon}\}_{\epsilon>0}\subset\mathcal{C}^{\infty}_{c}$ be an approximation to the identity and put $\mu_{\epsilon}=\psi_{\epsilon}*\mu,$ and $f_{\epsilon}=\psi_{\epsilon}*f$. By the convolution theorem,
\[\widehat{\mu}_{\epsilon}=\widehat{\psi}_{\epsilon}\muhat=\widehat{\psi}_{\epsilon}\fhat=\widehat{f}_{\epsilon},\  \varrho-\text{almost everywhere },\] 
and since $\mu_{\epsilon}, f_{\epsilon}\in L^2(\Hn)$, one has that $\mu_{\epsilon}=f_{\epsilon}$ almost everywhere. Let $\psi$ be a continuous bounded function on $\Hn$, then
\begin{align*}
\int_{\Hn}\psi(p)d\mu(p)&=\lim_{\epsilon\to0}\int_{\Hn}\psi(p)\mu_{\epsilon}(p)dp\\
&=\lim_{\epsilon\to0}\int_{\Hn}\psi(p)f_{\epsilon}(p)dp\\
&=\int_{\Hn}\psi(p)f(p)dp,
\end{align*}
which proves the claim.
\item If $\muhat\in L^1(\Sa_1)$ then, $|tr( \rhol(p)^*\muhat(\lambda))|\leq tr(| \rhol(p)^*\muhat(\lambda)|)\leq tr(|\muhat(\lambda)|)$. It follows that 
\[g(p):=\int_{\R^*}tr( \rhol(p)^*\muhat(\lambda))d \varrho(\lambda)\]
is a well defined continuous function on $\Hn$. As before, the aim is to show that $g$ is the density of $\mu$. To see this, let $\mu_{\epsilon}$ be as before. Since $\mu_{\epsilon}\in\Sw(\Hn)$, $\widehat{\mu}_{\epsilon}(\lambda)\in S_1$ for each $\lambda$. Also, $\widehat{\mu}_{\epsilon}(\lambda)\to\muhat(\lambda)$ in the strong operator topology. So by Lemma \ref{lemmaB}, there is a subsequence $\{\mu_k\}$ such that
 \begin{enumerate}[(i)]
 \item $\mu_k\in\Sw(\Hn)$,
 \item $\mu_k\to \mu$ weakly,
 \item $\muhat_k(\lambda)\to\muhat(\lambda)$ in the strong operator topology,
 \item and $tr( \rhol(p)^*\muhat_k(\lambda))\to tr( \rhol(p)^*\muhat(\lambda)).$
\end{enumerate}
By Theorem $\ref{inversion}$,
\[\mu_k(p)=\int_{\R^*}tr( \rhol(p)^*\muhat_k(\lambda))d \varrho(\lambda),\]
but also, 
\[|tr(\rhol(p)^*\muhat_k(\lambda))|\leq||\rhol(p)^*||_{op}||\widehat{\psi}_k(\lambda)||_{op}tr(|\muhat(\lambda)|)\lesssim tr(|\muhat(\lambda)|)\in L^1(S_1).\]
By the dominated convergence theorem, we have a point-wise limit
\begin{align*}
\lim_{k\to\infty}\mu_k(p)&=\lim_{k\to\infty}\int_{\R^*}tr( \rhol(p)^*\muhat_k(\lambda))d \varrho(\lambda)\\
&=\int_{\R^*}tr( \rhol(p)^*\muhat(\lambda))d \varrho(\lambda)=g(p).
\end{align*}
Since, for all $k\in\N$, $||\psi_k||_{L^1(\Hn)}\leq 1$, it follows that pointwise convergence, implies weak convergence. That is, for any continuous bounded function $\psi$,

\[\int_{\Hn}\psi(p)d\mu(p)=\lim_{k\to\infty}\int_{\Hn}\psi(p)\mu_k(p)dp=\int_{\Hn}\psi(p)g(p)dp. \]
This completes the proof.
\end{enumerate}
\end{proof}
\section{Fourier Coefficients in $\Hn$}\label{FCoeffH}

In the Heisenberg group, energy integrals of measures take the form
\[I_{\sigma}(\mu)=\iint||q^{-1}p||_{\Hn}^{-\sigma}d\mu(q)d\mu(p).\]
The inner integral takes the form 
\begin{equation}\label{KorEn}
\mu\ast K_{\sigma}(p),
\end{equation}
 where ``$\ast$" denotes group convolution, and $K_{\sigma}$ denotes the Riesz kernel for the Koranyi norm, $K_{\sigma}(p)=||p||_{\Hn}^{-\sigma}$ for $0\leq\sigma<2n+2$. One would like to use the convolution theorem to compute energy integrals in the frequency domain. However, $K_{\sigma}$ is not in $L^p$ for any $p$, therefore first step in such a computation would be to extend the group Fourier transform to tempered distributions. This has been recently done by the authors in \cite{BCD1} and \cite{BCD2} where they introduced a frequency space for $\Hn$ and, with it, an alternative definition for the Fourier transform in the Heisenberg group.  This section will present this approach and all results that will be relevant to the proof of the energy formula. These will be presented without proof but all the details are contained in the original references.\\
The main idea behind these Fourier coefficients is, instead of studying the operator $\fhat$, to study its matrix coefficients. The matrix coefficients of an operator, are dependent on the choice of an orthonormal basis, however in this case, it turns out to be convenient to choose a basis comprised of  eigenfunctions for the Euclidean Fourier transform.

\begin{definition}[Hermite functions and re-scaled Hermite functions.]
For $x\in\R$ and $k\in\N$, the $k$-th Hermite polynomial is the polynomial
\[\mathscr{H}_k(x):=(-1)^{k}\left[\frac{d^k}{dx^k}\{e^{-x^2}\}e^{x^2}\right].\]
The (normalized) $k$-Hermite function is 
\[\phi_{k}(x):=(2^k\sqrt{\pi}k!)^{-\frac{1}{2}}e^{-\frac{x^2}{2}}\mathscr{H}_k(x).\]
In higher dimensions, for $x\in\Rn$ and multi-index $\alpha\in\N^n$,
\[\phi_{\alpha}(x):=\Pi_{j=1}^{n}\phi_{\alpha_j}(x_j).\]
Finally, the re-scaled Hermite functions are given for $\alpha\in\N^n,\ \lambda\in\R^*,$ and $x\in\Rn$, by
\[\pla(x):=|\lambda|^{n/4}\phi_{\alpha}(|\lambda|^{1/2}x).\]
\end{definition}

 Let $\Ht=\N^n\times\N^n\times\R^*$, and denote a typical point by $\zeta=(\alpha,\beta,\lambda)$. For $f\in L^1(\Hn)$ let $\mathcal{F}_{\Hn}(f):\Ht\to\C$ be the map
 \begin{equation}\label{AltF}
\mathcal{F}_{\Hn}(f)(\zeta)=\mathcal{F}_{\Hn}(f)(\alpha,\beta,\lambda):=\langle\fhat(\lambda)\pla,\plb\rangle_{L^2(\Rn)}.
 \end{equation}
 For simplicity, $\mathcal{F}_{\Hn}(f)$ is interchangeably denoted by $ \fhatH.$ This transform shares many important properties, or at least analogues thereof, with the classical Fourier transform. Some of these properties are what allowed the authors to extend it to tempered distribution by characterizing the image of $\Sw(\Hn)$. Many of these properties will be of great relevance later in this work when computing energies of measures via integrals on the frequency domain. Because of this, a brief introduction to this Fourier coefficient approach is given next. Many of the proofs are skipped, but readers are referred to the aforementioned references where all the proofs are presented and more details are given.
 
Perhaps the most clear advantage of this approach is that the function $\fhatH$ can be written as an integral of $f$ with respect to an appropriate kernel. Specifically, letting $\Phi:\Hn\times\Ht\to\C$ be given by
 \begin{equation}\label{PHI}
 \Phi(p,\zeta)=\langle\rhol(p)\pla,\plb \rangle_{L^2(\Rn)}=e^{i\lambda t}\int_{\Rn} e^{i\lambda x\dotp\xi}\pla(\xi+\frac{y}{2})\plb(\xi-\frac{y}{2})d\xi,
 \end{equation}
a quick computation using Fubini's theorem shows that
\begin{equation}\label{FInt}
\fhatH(\zeta)=\int_{\Hn}\Phi(p,\zeta)f(p)dp.
\end{equation} 
It is worth mentioning that the integral $\int_{\Rn} e^{i\lambda x\dotp\xi}\pla(\xi+\frac{y}{2})\plb(\xi-\frac{y}{2})d\xi$ is what is known as the Fourier-Wigner transform of $\pla$ and $\plb$. In general, the Fourier-Wigner transform of two functions $\varphi,\psi \in L^2(\Rn)$ is the function on $\Rnn$ given by
\[V_{\lambda}(f,g)(z):=\langle\rhol(z,0)f,g\rangle_{L^2(\Rn)}=\int_{\Rn}e^{i\lambda x\dotp\xi}f(\xi+\frac{y}{2})\overline{g(\xi-\frac{y}{2})}d\xi.\]

 This integral formulation of $\FH$ simplifies many computations. For instance, the inversion formula for $\Sw(\Hn)$ becomes much simpler. Before we state it, the space $\Ht$ must be endowed with a measure. For $\eta:\Ht\to\C$, let $d\zeta$ be the measure given by
\begin{equation}
\int_{\Ht}\eta(\zeta)d\zeta=\sum_{\alpha\in\N^n}\sum_{\beta\in\N^n}\int_{\R^*}\eta(\alpha,\beta,\lambda)d\rho(\lambda).
\end{equation}
One has, for $f\in\Sw(\Hn)$ ,
\[f(p)=\mathcal{F}_{\Hn}^{-1}(\fhatH)(p)=\int_{\Ht}\overline{\Phi(p,z)}\fhatH(\zeta)d\zeta.\]
One of the most important properties of the Fourier transform is its effect on convolutions. With this approach the convolution theorem recasts as follows,
\begin{equation}\label{coeffconv}
\mathcal{F}_{\Hn}(f*g)(\alpha,\beta,\lambda)=\sum_{\gamma\in\N^n}\fhatH(\gamma,\beta, \lambda)\widehat{g}_{\Hn}(\alpha,\gamma,\lambda)=:\fhatH\cdot \widehat{g}_{\Hn}(\alpha,\beta,\lambda).
\end{equation}
Yet another property of the map $\FH$ is that it turns smoothness into decay. This was precisely the content of Lemma \ref{decayF} which we now restate in the notation of this section.

\begin{lemma}
 For any $q\in\mathbb{N}$, there exist $N_q\in\mathbb{N}$ and constant $C_q>0$ such that
 \begin{equation}\label{decay2}
 \left[ 1+|\lambda|(|\alpha|+|\beta|+n)+|\alpha-\beta| \right]^q|\fhatH(\alpha,\beta,\lambda)|\leq C_q||f||_{N_q,\Sw},
 \end{equation}
 where $||f||_{N_q,\Sw}$ is the classical family of Schwartz semi-norms,
 \[||f||_{N,\Sw(\Hn)}:=\sup_{|\alpha|\leq N}||(1+|z|^2+t^2)^{N/2}\partial^{\alpha}_{z,t}f||_{L^{\infty}(\Hn)}.\]
\end{lemma}
This decay result motivates the definition of the metric $\dht$ on $\Ht$ defined by 
\begin{equation}\label{metric}
\dht(\zeta,\zeta')=|\lambda(\alpha+\beta)-\lambda'(\alpha'+\beta')|_{\ell^1}+|(\alpha-\beta)-(\alpha'-\beta')|_{\ell^1}+|\lambda-\lambda'|.
\end{equation}
The metric $\dht$ will be relevant later. Now we focus on the characterization of $\FH(\Sw(\Hn))$, which requires several definitions and results first.

\begin{definition}
Let $\eta:\Ht\to\C$ be a function differentiable with respect to $\lambda$, and $f$ be a function in $\Sw(\Hn)$. For $j\in\N$, denote by $\delta_{j}\in\N^n$ the point with a ``1" in the $j^{th}$ coordinate and zeros elsewhere, and set $\zeta^{\pm}_j=(\alpha\pm\delta_j,\beta\pm\delta_j,\lambda)$. Define the following six operators,
 \begin{equation}
\widehat{\Delta}\eta(\zeta):=\frac{1}{|\lambda|}(|\alpha+\beta|+n)\eta(\zeta)+\frac{1}{|\lambda|}\sum_{j=1}^{n}(\sqrt{\alpha_j\beta_j}\eta(\zeta^{-}_j)+\sqrt{(\alpha_j+1)(\beta_j+1)}\eta(\zeta^{+}_j)),
\end{equation}

 \begin{equation}
\widehat{D}_{\lambda}\eta(\zeta):=\frac{d\eta}{d\lambda}(\zeta)+\frac{n}{\lambda}\eta(\zeta)+\frac{1}{|\lambda|}\sum_{j=1}^{n}(\sqrt{\alpha_j\beta_j}\eta(\zeta^{-}_j)-\sqrt{(\alpha_j+1)(\beta_j+1)}\eta(\zeta^{+}_j) ),
\end{equation}

\begin{equation}
\mathcal{P}f(p):=\int_{-\infty}^t(f(z,s)-f(z,-s))ds
\end{equation}

\begin{equation}
\widehat{\Sigma}_{0}(\eta):=\frac{\eta(\alpha,\beta,\lambda)-(-1)^{|\alpha+\beta|}\eta(\alpha,\beta,-\lambda)}{\lambda}
\end{equation}

\begin{equation}
(M^2f)(z,t):=|z|^2f(z,t)
\end{equation}

\begin{equation}
(M_0f)(z,t):=itf(z,t).
\end{equation}

\end{definition}
Then, the following holds
\begin{theorem} For $f\in\Sw(\Hn)$,
\begin{align}
&\FH(M^2f)=-\widehat{\Delta}(\FH f)\\
&\FH(M_0f)=\widehat{D}_{\lambda}(\FH f)\\
&\FH(\mathcal{P}f)=-i\widehat{\Sigma}_0(\FH f)
\end{align}
\end{theorem}
In a sense, this is analogous to the fact that the Fourier transform turns decay into smoothness. With this at hand one may define the Schwartz space on $\Ht$ as follows,
\begin{definition}
$\eta\in\Sw(\Ht)$ if 
\begin{enumerate}
\item For any given $(\alpha,\beta)\in\N^n$, $\eta(\alpha,\beta,\cdot):\R^*\to\C$ is smooth.
\item For any $N\in\N$ the functions $\widehat{\Delta}^N\eta,\ \widehat{D}_{\lambda}^N\eta,\text{ and } \widehat{\Sigma}_0\widehat{D}_{\lambda}^N\eta$, all decay faster than any power of $d_0=|\lambda|(|\alpha+\beta|+n)+|\alpha-\beta|$.
\end{enumerate}
\bigskip 
The space $\Sw(\Ht)$ is equipped with the family of semi-norms 
\begin{equation}\label{SwSemiNorm}
||\eta||_{N,N',\Sw(\Ht)}:=\sup_{\zeta\in\Ht}(1+d_0(\zeta))^N[\widehat{\Delta}^{N'}\eta + \widehat{D}_{\lambda}^{N'}\eta + \widehat{\Sigma}_0\widehat{D}_{\lambda}^{N'}\eta].\end{equation}
\end{definition}
\begin{theorem}\label{inversionC}
The map $\FH:\Sw(\Hn)\to\Sw(\Ht)$ is a a continuous isomorphism with continuous inverse given by \[\FH^{-1}\eta(p)=\int_{\Ht}\overline{\Phi(p,\zeta)}\eta(\zeta)d\zeta.\]
\end{theorem}
Perhaps the biggest inconvenience of the metric space $(\Ht,\dht)$ is that it is not complete, however, for any $f\in L^1(\Hn)$, $\fhatH$ is uniformly continuous on $(\Ht,\dht)$. It is therefore natural to extend $\fhatH$ to the metric completion of $(\Ht,\dht)$ which is denoted by $(\Hhat,\dhat)$. 
The metric space $(\Hhat,\dhat)$ is explicitly given by $\Hhat=\Ht\cup\Hhat_0$, where $\Hhat_0=\Rn_{\mp}\times\Z^n$, with $\Rn_{\mp}=(\R_-)^n\cup(\R_+)$. For $\zeta\in\Hhat$, if $\zeta\in\Hhat_0$ we denote it by $\zeta=(\dot{x},k)$. The metric $d_{\Hhat}$ is given by

\begin{align*}
&d_{\Hhat}(\zeta,\zeta')=d_{\Ht}(\zeta,\zeta'), && \text{ if }\  \zeta,\zeta'\in\Ht\\
&d_{\Hhat}((\alpha,\beta,\lambda),(\dot{x},k))=|\lambda(\alpha+\beta)-\dot{x}|_{\ell^1}+|\alpha-\beta-k|_{\ell^1}+|\lambda|, && \text{ if }\ (\alpha,\beta,\lambda)\in\Ht,\ (\dot{x},k)\in\Hhat_0\\
&d_{\Hhat}((\dot{x},k),(\dot{x}',k'))=|\dot{x}-\dot{x}'|_{\ell^1}+|k-k'|_{\ell^1}, && \text{ if }\ (\dot{x},k),(\dot{x}',k')\in\Hhat_0.
\end{align*} 
Abusing notation, the extension of $\fhatH$ is also denoted by $\fhatH$. The space $\Sw(\Hhat)$ is defined to be the space of all continuous functions, $\eta$, in $\Hhat$ such that $\eta\lfloor_{\Ht}\in\Sw(\Ht)$. It is not hard to see that 
\[\int_{\Ht}\mathbbm{1}_{\{|\lambda||\alpha+\beta|+|\alpha-\beta|\leq R\}}(\zeta)\mathbbm{1}_{\{|\lambda|\leq \epsilon\}}(\zeta)d\zeta\lesssim R^{2n}\epsilon.\]
Therefore, it is natural to extend $d\zeta$ to $\Hhat$ by defining, for all $\psi\in\mathcal{C}_c(\Hhat)$,
\[\int_{\Hhat}\psi(\zeta)d\zeta:=\int_{\Ht}\psi\lfloor_{\Ht}(\zeta)d\zeta.\]
The convolution theorem extends as follows,
\[\FH(f*g)(\dot{x},k)=\sum_{k'\in\Z^n}\fhatH(\dot{x},k-k')\widehat{g}_{\Hn}(\dot{x},k').\]

The space $\Sw(\Hhat)$ with the semi-norms \eqref{SwSemiNorm} is a Frech\'{e}t space, so It makes sense to talk about its topological dual. This will be the class of tempered distributions on $\Ht$ and will be denoted $\Sw'(\Hhat)$. This allows the following definition.
 \begin{definition}\label{Ftemper}
 For $T\in\Sw'(\Hn)$, $\FH(T)\in\Sw'(\Hhat)$ is given by
 \[\langle\FH T|\eta\rangle_{\Sw'(\Hhat)\times\Sw(\Hhat)}=\langle T| \FH^{\tau}\eta\rangle_{\Sw'(\Hn)\times\Sw(\Hn)}.\]
 Here $\FH^{\tau}:\Sw(\Hhat)\to\Sw(\Hn)$ is the map
 \[\FH^{\tau}\eta(p):=\int_{\Hhat}\Phi(p,\zeta)\eta(\zeta)d\zeta.\]
 \end{definition}
 Turns out $\FH:\Sw'(\Hn)\to\Sw'(\Hhat)$ is a continuous injection.

 \subsection{Fourier coefficients of measures}

This distributional definition of $\FH$ clearly applies to measures in $\Ma(\Hn)$. The goal of this section is to extend the study of properties of this distributional transform keeping in mind the goal of computing the Fourier transform of $K_{\sigma}$. This is precisely a first step in obtaining a frequency formulation of energy integrals. In addition to the distributional definition of $\FH$, the definition of $\FH$ on functions can be extended to $\Ma(\Hn)$ by using either \eqref{AltF} or \eqref{FInt}. A quick use of Fubini's theorem shows that, just as with functions, these give equivalent definitions. This gives us a point-wise definition of $\muhatH$ for $\mu\in\Ma(\Hn)$,
 \begin{equation}\label{muhat}
 \muhatH(\zeta)=\int_{\Hn}\Phi(p,\zeta)d\mu(p).
 \end{equation}

 \begin{lemma}\label{muhatd}
Let $\muhatH$ be as in \eqref{muhat}, and let $\FH\mu$ denote the distributional Fourier transform  $\mu$ as in Definition \ref{Ftemper}. Then, $\FH\mu=\muhatH$ as distributions. That is to say, $\FH\mu$ is given by integration against the function $\muhatH$.
 \end{lemma}
 \begin{proof}
 The proof is a straight forward computation using Fubini's theorem. Let $\eta\in\Sw(\Hhat)$, then
 \begin{align*}
 \langle\FH\mu|\eta\rangle&=\langle\mu|\FH^{\tau}\eta\rangle\\
 &=\int_{\Hn}\int_{\Ht}\Phi(p,\zeta)\eta(\zeta)d\zeta d\mu(p)\\
 &=\int_{\Ht}\eta(\zeta)\int_{\Hn}\Phi(p,\zeta)d\mu(p)d\zeta=\int_{\Ht}\eta(\zeta)\muhatH(\zeta)d\zeta.
 \end{align*}
 \end{proof}
 The following lemma, which before took some work to prove, is now immediate.
 \begin{lemma}
 For $f\in\Sw(\Hn)$, $\mu\in\Ma(\Hn)$
 \[\int_{\Hn}f(p)d\mu(p)=\int_{\Ht}\fhatH(\zeta)\overline{\muhatH(\zeta)}d\zeta.\]
 \end{lemma}
 \begin{proof}
 Once again the proof is a simple computation involving Fubini's theorem and Theorem \ref{inversionC},
 \begin{align*}
 \int_{\Hn}f(p)d\mu(p)&=\int_{\Hn}[\FH^{-1}\fhatH](p)d\mu(p)\\
 &=\int_{\Hn}\int_{\Ht}\overline{\Phi(p,\zeta)}\fhatH(\zeta)d\zeta d\mu(p)\\
 &=\int_{\Ht}\fhatH(\zeta)\int_{\Hn}\overline{\Phi(p,\zeta)}d\mu(p)=\int_{\Ht}\fhatH(\zeta)\overline{\muhatH(\zeta)}d\zeta.
 \end{align*}
 \end{proof}
 The hope is to extend this Lemma to the integral in \eqref{KorEn}. However, while this applies to $f\in\Sw(\Hn)$, $K_{\sigma}*\mu$ is not in the Schwarz class and, a priori, may not even be a tempered distribution. Moreover, as of yet, there is no suitable extension of \eqref{coeffconv} to distributions. If $f,g,h\in\Sw(\Hn)$ a quick computation yields 
 \[\int_{\Hn}f*g(p)h(p)dp=\int_{\Hn}g(q)\tilde{f}*h(q)dq,\]
 where $\tilde{f}(p)=f(-p)$. This motivates the definition of the convolution of distributions with Schwartz functions. We recall the definition here,
 \begin{definition}
 For $g\in\Sw(\Hn)$, $T\in\Sw'(\Hn)$, $g*T$ is the distribution given, for all $f\in\Sw(\Hn)$, by
 \[\langle g*T|f \rangle=\langle T|\tilde{g}*f\rangle.\]
 \end{definition}
 
 In a similar way, for $\eta,\theta,\psi\in\Sw(\Hhat)$ we compute
 \begin{align*}
 \int_{\Hhat}\eta\cdot\theta(\zeta)\psi(\zeta)d\zeta&=\int_{\Ht}\eta\cdot\theta(\zeta)\psi(\zeta)d\zeta\\
 &=\sum_{\alpha,\beta}\int_{\R^*}\sum_{\gamma}\eta(\gamma, \beta, \lambda)\theta(\alpha, \gamma,\lambda)\psi(\alpha,\beta,\lambda)d\rho(\lambda)\\
 &=\sum_{\gamma,\alpha}\int_{\R^*}\theta(\alpha, \gamma,\lambda)\sum_{\beta}\eta(\gamma,\beta,\lambda)\psi(\alpha,\beta\lambda)d\rho(\lambda)\\
 &=\int_{\Ht}\theta(\zeta)(\eta_{\tau}\cdot\psi)(\zeta)d\zeta=\int_{\Hhat}\theta(\zeta)(\eta_{\tau}\cdot\psi)(\zeta)d\zeta,
 \end{align*}
 
 where $\eta_{\tau}(\alpha,\beta,\lambda)=\eta(\beta,\alpha,\lambda)$, and $\eta_{\tau}(\dot{x},k)=\eta(\dot{x},-k)$. This motivates the following definition.
 \begin{definition}
 For $\eta\in\Sw(\Hhat)$ and $\Psi\in\Sw'(\Hhat)$, $\eta\cdot\Psi$ is defined as the distribution given, for all $\theta\in\Sw(\Hhat)$, by,
 \[\langle\eta\cdot\Psi|\theta \rangle=\langle\Psi|\eta_{\tau}\cdot\theta \rangle\]
 \end{definition}

 Along side the transform $\FH^{\tau}$ the following transform is also introduced: $\FH^{-\tau}:\Sw(\Hn)\to\Sw(\Hhat)$, given by \[\FH^{-\tau}f(\zeta)=\int_{\Hn}\overline{\Phi(p,\zeta)}f(p)dp.\]
 With this, we have a total of 4 ``Fourier-like" transforms $\FH,\FH^{-\tau}:\Sw(\Hn)\to\Sw(\Hhat)$ and $\FH^{-1},\FH^{\tau}:\Sw(\Hhat)\to\Sw(\Hn)$. The following relations between these are easily verified.
 \begin{lemma}\label{FTRelations}
 Let $f\in\Sw(\Hn)$, $\eta\in\Sw(\Hhat)$ and for any function $\vartheta:\Hhat\to\C$ denote $\vartheta_-(\alpha,\beta,\lambda)=\vartheta(\alpha,\beta,-\lambda)$.
 \begin{enumerate}
 \item  $(\FH)^{-1}=\FH^{-1}$,
 \item $(\FH^{\tau})^{-1}=\FH^{-\tau}$
 \item $\FH^{-\tau}f=[\FH f]_-$
 \item $\FH^{-1}\eta=\FH^{\tau}(\eta_-)$
 \item $\FH^{-\tau}\tilde{f}(\zeta)=[\FH f]_{\tau}(\zeta)$
  \end{enumerate}
 \end{lemma}

 \begin{lemma}\label{ConvFDist}
 For $f\in\Sw(\Hn)$, and $T\in\Sw'(\Hn)$, $\FH(f*T)=\fhatH\cdot\widehat{T}_{\Hn}$ in the sense of distributions.
 \end{lemma}
 \begin{proof} 
 Let $f\in\Sw(\Hn)$, $T\in\Sw'(\Hn)$ and $\eta\in\Sw(\Hhat)$. Then, using (5) on Lemma \ref{FTRelations},
 \begin{align*}
  \langle\FH(f*T)|\eta \rangle &=\langle f*T|\FH^{\tau}\eta \rangle\\
  &=\langle T|\tilde{f}*\FH^{\tau}\eta\rangle\\
  &=\langle T|\FH^{\tau}(\FH^{-\tau}\tilde{f}\cdot\eta) \rangle\\
  &=\langle\FH T|\FH^{-\tau}\tilde{f}\cdot\eta \rangle\\
&=\langle\FH T|[\FH{f}]_{\tau}\cdot\eta \rangle\\
&=\langle \FH{f}\cdot\FH T|\eta \rangle.
\end{align*}
This completes the proof.
 \end{proof}
Let's finishing this section by computing  $\FH\delta_0$ in the sense of distributions, where $\delta_0$ is the Dirac distribution at zero. This simple computation was done by the authors on \cite{BCD2} but since it will be of great relevance when computing the Fourier transform of $K_{\sigma}$, we do it here as well. For $\eta\in\Sw(\Hn)$,
\begin{align*}
\langle \FH\delta_0|\eta\rangle&=\langle \delta_0|\FH^{\tau}\eta\rangle\\
&=\FH^{\tau}\eta(0)=\int_{\Ht}\Phi(0,\zeta)\eta(\zeta)d\zeta\\
&=\int_{\Ht}\int_{\Rn}\pla(\xi)\plb(\xi)d\xi\eta(\zeta)d\zeta.
\end{align*} 
Since $\{\pla\}_{\alpha\in\N^n}$ is an orthonormal basis for $L^{2}(\Rn)$ it follows that the inner integral is $\mathbbm{1}_{\{\alpha=\beta\}}$. Therefore,
\[\langle \FH\delta_0|\eta\rangle=\int_{\Ht}\mathbbm{1}_{\{\alpha=\beta\}}(\zeta)\eta(\zeta)d\zeta.\]
It follows that $\FH\delta_0=\mathbbm{1}_{\{\alpha=\beta\}}$ in the sense of distributions.
 \subsection{The Fourier transform of the Koranyi-Riesz kernel}
 
 Another improtant property of the classical Fourier transform is its interaction with differential operators. In $\Hn$ one should expect any reasonable definition of the Fourier transform to interact nicely with the subLaplacian operator 
 \[\Delta_{\Hn}=\sum_{j=1}^nX_j^2+Y_j^2.\] 
Turns out that $\Delta_{\Hn}$ is a $\FH$ multiplier with $\FH\Delta_{\Hn}f=|\lambda|(2|\alpha|+n)\FH f$, $f\in\Sw(\Hn)$. In \cite{Folland3}, G.B. Folland showed that there is a constant $d=d(n)$ depending only on $n$ such that the fundamental solution of $\Delta_{\Hn}$ is $d(n)^{-1}K_{Q-2}$. This allows to, quite easily, compute $\FH K_{Q-2}$ in the sense of distributions. This computation will be done later in greater generality (i.e. for $K_{\sigma},\ 0<\sigma<Q$) but it turns out that 

 \[\FH K_{Q-2}(\alpha,\beta,\lambda)=d(n)\frac{1}{|\lambda|(2|\alpha|+n)}\mathbbm{1}_{\alpha=\beta}(\alpha,\beta,\lambda).\]
  
This gives an explicit formula for $\FH K_{\sigma} $ for a particular choice of $\sigma$. To obtain the more general case of arbitrary $0\leq \sigma\leq Q$, one might be tempted to consider a fundamental solution of the fractional subLaplacian $\Delta_{\Hn}^{\sigma/2}$, and use a similar computation. One can show that if $R_{\sigma}$ is a fundamental solution of $\Delta_{\Hn}^{\sigma/2}$, then $\FH R_{\sigma}=\tilde{d}(n,\sigma)\frac{\mathbbm{1}_{\alpha=\beta}}{[|\lambda|(2|\alpha|+n)]^{\sigma/2}}$ in the sense of distributions. However it is not true that $R_{\sigma}$ is a constant multiple of $K_{\sigma}$ (unless, of course, $\sigma=2$). In fact, there is no known explicit formula for $R_{\sigma}$. There has been work done in the direction of realizing $R_{\sigma}$ as explicitly as possible, see for instance \cite{BDR}, and more recently \cite{WangWu} where a specific case is treated. Instead, it proves useful to consider conformally invariant fractional powers of the subLaplacian. This operator arises naturally in CR geometry as studied for instance in  \cite{BFM}, \cite{BOO}, \cite{FrankLieb} and \cite{JohaWallch}. It also arises in studying the extension problem in $\Hn$, \cite{FrankGonzMontiTan}, and has been studied in the context of Hardy inequalities in the Heisenberg group, \cite{RoncalThang}. In what follows, this operator will be introduced and studied in the context of $\FH$. \\

The function $\Phi$, introduced before, is of the form $e^{i\lambda t}|\lambda|^{-\frac{n}{2}}(2\pi)^{\frac{n}{2}}\Theta$, where
\[\Theta(z,\zeta)=\left(\frac{|\lambda|}{2\pi}\right)^{\frac{n}{2}}\int_{\Rn}e^{i\lambda x\cdot\xi}\pla(\xi+\frac{y}{2})\plb(\xi-\frac{y}{2})d\xi=\left(\frac{|\lambda|}{2\pi}\right)^{\frac{n}{2}}\langle\rhol(z,0)\pla|\plb \rangle,\]
are the rescaled special Hermite functions. As mentioned before, this is also the Fourieri-Wigner transform of $\pla$ and $\plb$. The set $\{\Theta(\cdot,\zeta)\}_{\alpha,\beta\in\N^n}$ forms a complete orthonormal basis for $L^2(\C^n)$. In particular, for $f\in\Sw(\Hn)$
\begin{equation}\label{SpecHermExp}
f^{\lambda}(z)=\sum_{\alpha}\sum_{\beta}\langle f^{\lambda}|\Theta(\cdot, \zeta) \rangle\Theta(z,\zeta).
\end{equation}

This expansion has a more compact form in terms of Laguerre functions. For $k\in\N$, the $\text{k}^{th}$ Laguerre polynomial of type $\delta>-1$ is 
\[L_k^{\delta}(t)e^{-t}t^{\delta}=\frac{1}{k!}(\frac{d}{dt})^k(e^{-t}t^{k+\delta}).\]
The Laguerre functions are given in terms of $L_k^{\delta}$ by 
\[\ell_{k}^{\delta}(z):=L_{k}^{\delta}(\frac{1}{2}|z|^2)e^{-\frac{1}{2}|z|^2}.\]
Just as with the Hermite functions, $\ell_k^{\delta}$ can be adapted to the representation theory of $\Hn$ by a simple rescaling. The rescaled Laguerre function is given by $\ell_{k,\lambda}^{\delta}(z)=\ell_k^{\delta}(\sqrt{|\lambda|}z)$. This functions are closely related to the functions $\Theta(z,\zeta)$.
\begin{theorem}\label{HermLag}
\[\sum_{|\alpha|=k}\Theta(z,\alpha,\alpha, \lambda)=(2\pi)^{-\frac{n}{2}}(|\lambda|)^{\frac{n}{2}}\ell_{k,\lambda}^{n-1}(z).\]
\end{theorem}
For a proof of this the reader is referred to \cite{Tangavelu2} (see in particular 2.3.26). Now consider the convolution
\[f*\Phi(\cdot,\zeta_{\alpha})(p)=\int_{\Hn}f(q)\Phi(q^{-1}p,\zeta_{\alpha})dq,\]
where $\zeta_{\alpha}=(\alpha,\alpha,\lambda)\in\Ht$. Since 
\[\Phi(q^{-1}p,\zeta)=e^{i\lambda(t-s+\frac{1}{2}Im(z\overline{w}))}(2\pi)^{\frac{n}{2}}|\lambda|^{-\frac{n}{2}}\Theta(z-w,\zeta),\]
it follows that ,
\[f*\Phi(\cdot,\zeta_{\alpha})(p)=(2\pi)^{\frac{n}{2}}|\lambda|^{-\frac{n}{2}}e^{i\lambda t}\int_{\C^n}f^{-\lambda}(w)\Theta(z-w,\zeta_{\alpha})e^{\frac{i\lambda}{2}Im(z\overline{w})}dw=(2\pi)^{\frac{n}{2}}|\lambda|^{-\frac{n}{2}}e^{i\lambda t}(f^{-\lambda}*_{\lambda}\Theta(\cdot,\zeta_{\alpha}))(z).\]
Here, $*_{\lambda}$ denotes $\lambda-$twisted convolutions, given for $g,h\in L^1(\C^n)$ by 
\[g*_{\lambda}h(z)=\int_{\C^n}g(w)h(z-w)e^{\frac{i\lambda}{2}Im(z\overline{w})}.\] 
Now, Since $\rhol(p)\pla$ is in $L^2(\Rn)$ it follows that $\rhol(p)\pla=\sum_{\beta}\langle \rhol(p)\pla|\plb \rangle\plb$. Therefore
\begin{align*}
\Phi(q^{-1}p,\zeta_{\alpha})&=\langle\rhol(q^{-1})\sum_{\beta}\langle \rhol(p)\pla|\plb \rangle\plb|\pla \rangle\\
&=\sum_{\beta}\langle \rhol(p)\pla|\plb \rangle\langle\rhol(q^{-1})\plb|\pla\rangle\\
&=\sum_{\beta}\Phi(p,\zeta)\overline{\Phi(q,\zeta)}.
\end{align*}
Which gives
\begin{align*}
(f*\Phi(\cdot,\zeta_{\alpha}))(p)=\sum_{\beta}\int_{\Hn}f(q)\overline{\Phi(q,\zeta)}dq\Phi(p,\zeta)=(2\pi)^{n}|\lambda|^{-n}e^{i\lambda t}\sum_{\beta}\langle f^{-\lambda}|\Theta(\cdot,\zeta)\rangle\Theta(z,\zeta).
\end{align*}
With this, and Theorem \ref{HermLag}, \eqref{SpecHermExp} rewrites as
\begin{equation}
\begin{split}
f^{\lambda}(p)&=\frac{|\lambda|^{n}}{(2\pi)^n}\sum_{k=0}^{\infty}\sum_{|\alpha|=k}e^{-i\lambda t}(f*\Phi(\cdot,\zeta))(p)=\frac{|\lambda|^{n/2}}{(2\pi)^{n/2}}\sum_{k=0}^{\infty}\sum_{|\alpha|=k}(f^{-\lambda}*_{\lambda}\Theta(\cdot,\zeta_{\alpha}))(z)\\&=\frac{|\lambda|^{n}}{(2\pi)^n}\sum_{k=0}^{\infty}(f^{-\lambda}*_{\lambda}\ell_{k,-\lambda}^{n-1})(z)=\frac{|\lambda|^{n}}{(2\pi)^n}\sum_{k=0}^{\infty}(f^{\lambda}*_{-\lambda}\ell_{k,\lambda}^{n-1})(z).
\end{split}
\end{equation}
This relation between the special Hermite expansion and the Laguerre expansion will be instrumental in computing $\mathfrak{K}_{\sigma}$. First, conformally invariant fractional powers of the subLaplacian are introduced.
\begin{definition}
For $0\leq \sigma\leq Q$ the conformally invariant $\sigma$-fractional subLaplacian is the operator $\mathcal{L}_{\sigma}:\Sw(\Hn)\to\Sw(\Hn)$ given by 
\begin{equation}
\mathcal{L}_{\sigma}f(p)=\int_{-\infty}^{\infty}\sum_{k=0}^{\infty}(2|\lambda|)^{\sigma/2}\frac{\Gamma(k+\frac{Q+\sigma}{4})}{\Gamma(k+\frac{Q-\sigma}{4})}f^{\lambda}*_{\lambda}\ell_{k,\lambda}^{n-1}(z)e^{-i\lambda t}d\rho(\lambda).
\end{equation}
\end{definition}

This `modified' subLaplacian if of great importance here because its fundamental solution is precisely $d(n,\sigma)^{-1}K_{Q-\sigma}$ where $d(n,\sigma)$ is a constant depending only on $n$ and $\sigma$. A proof of this can be found in \cite[Section 3]{RoncalThang} (see, specifically, equation (3.10)). One can see that taking $\sigma=2$
\begin{align*}
(2|\lambda|)^{\sigma/2}\frac{\Gamma(k+\frac{Q+\sigma}{4})}{\Gamma(k+\frac{Q-\sigma}{4})}=2|\lambda|\frac{\Gamma(k+\frac{n}{2}+1)}{\Gamma(k+\frac{n}{2})}=2|\lambda|(k+\frac{n}{2})=|\lambda|(2k+n).
\end{align*}
This highlights the relation between $\mathcal{L}_{\sigma}$ and $\Delta_{\Hn}^{\sigma/2}$. Indeed, $\Delta_{\Hn}^{\sigma/2}=U(\sigma)\mathcal{L}_{\sigma}$ where $U(\sigma)$ is a bounded operator depending only on $\sigma$. In particular, $U(2)$ is the identity operator.

\begin{proposition}\label{multiplyer}
For $f\in\Sw(\Hn)$, $\mathcal{L}_{\sigma}$ satisfies 
\[ \mathcal{L}_{\sigma}f=\FH^{-1}[(2|\lambda|)^{\sigma/2}\frac{\Gamma(|\alpha|+\frac{Q+\sigma}{4})}{\Gamma(|\alpha|+\frac{Q-\sigma}{4})}\FH f].\]
\end{proposition}

\begin{proof}
From the previous discussion, it follows that $\mathcal{L}_{\sigma}f$ can be written as
\begin{equation}\label{A}
\int_{-\infty}^{\infty}\left[\sum_{\alpha,\beta}(2|\lambda|)^{\sigma/2}\frac{\Gamma(|\alpha|+\frac{Q+\sigma}{4})}{\Gamma(|\alpha|+\frac{Q-\sigma}{4})}(2\pi)^{n}|\lambda|^{-n}\langle f^{-\lambda}|\Theta(\cdot,\zeta)\rangle \Theta(z,\zeta)\right]e^{i\lambda t}d\rho(\lambda).
\end{equation}
One can also see that 
\[(2\pi)^{\frac{n}{2}}|\lambda|^{-\frac{n}{2}}\langle f^{-\lambda}|\Theta(\cdot,\zeta)\rangle=\int_{\Hn}f(q)\overline\Phi(q,\zeta)dq=\FH^{-\tau}f(\zeta).\]
By (3) on Lemma \ref{FTRelations} and by the change of variable $\lambda\to-\lambda$, \eqref{A} becomes 
\begin{align*}
\int_{\Ht}(2|\lambda|)^{\sigma/2}\frac{\Gamma(|\alpha|+\frac{Q+\sigma}{4})}{\Gamma(|\alpha|+\frac{Q-\sigma}{4})}&\FH f(\zeta) (2\pi)^{\frac{n}{2}}|\lambda|^{-\frac{n}{2}}\overline{\Theta(z,\zeta)}e^{-i\lambda t}d\zeta\\
&=\int_{\Ht}(2|\lambda|)^{\sigma/2}\frac{\Gamma(|\alpha|+\frac{Q+\sigma}{4})}{\Gamma(|\alpha|+\frac{Q-\sigma}{4})}\FH f(\zeta) \overline{\Phi(z,\zeta)}d\zeta\\
&=\FH^{-1}\left[(2|\lambda|)^{\sigma/2}\frac{\Gamma(|\alpha|+\frac{Q+\sigma}{4})}{\Gamma(|\alpha|+\frac{Q-\sigma}{4})}\FH f\right],
\end{align*}
as claimed.
\end{proof}

\begin{corollary}\label{multiplierSw}
For $f\in\Sw(\Hn)$, and $\zeta\in\Ht$, \[\FH\mathcal{L}_{\sigma}f(\zeta)=(2|\lambda|)^{\sigma/2}\frac{\Gamma(|\alpha|+\frac{Q+\sigma}{4})}{\Gamma(|\alpha|+\frac{Q-\sigma}{4})}\FH f(\zeta)\]
\end{corollary}
\begin{proof}
Follows trivially from Proposition \ref{multiplyer}.
\end{proof}

On the frequency side one has the following relation,
\begin{corollary}
For $\eta\in\Sw(\Hhat)$,
\[\mathcal{L}_{\sigma}\FH^{\tau}\eta=\FH^{\tau}\left[(2|\lambda|)^{\sigma/2}\frac{\Gamma(|\alpha|+\frac{Q+\sigma}{4})}{\Gamma(|\alpha|+\frac{Q-\sigma}{4})}\eta\right]\]
\end{corollary}
\begin{proof}
Using (3) on Lemma \ref{FTRelations} and Proposition \ref{multiplyer} one has,
\begin{align*}
\mathcal{L}_{\sigma}\FH^{\tau}\eta&=\FH^{-1}\left[(2|\lambda|)^{\sigma/2}\frac{\Gamma(|\alpha|+\frac{Q+\sigma}{4})}{\Gamma(|\alpha|+\frac{Q-\sigma}{4})}\FH\FH^{\tau}\eta\right]\\
&=\FH^{-1}\left[(2|\lambda|)^{\sigma/2}\frac{\Gamma(|\alpha|+\frac{Q+\sigma}{4})}{\Gamma(|\alpha|+\frac{Q-\sigma}{4})}\eta_{-}\right]\\
&=\FH^{\tau}\left[(2|\lambda|)^{\sigma/2}\frac{\Gamma(|\alpha|+\frac{Q+\sigma}{4})}{\Gamma(|\alpha|+\frac{Q-\sigma}{4})}\eta\right]
\end{align*}
\end{proof}

In order to compute $\FH K_{\sigma}$ explicitly, Corollary \ref{multiplierSw} needs to be extend to distributions. A quick computation using Lemma \ref{FTRelations} shows that for $f,g\in\Sw(\Hn)$
\begin{equation}
\int_{\Hn}\mathcal{L}_{\sigma}f(p)g(p)dp=\int_{\Hn}f(p)\mathcal{L}_{\sigma}g(p)dp.
\end{equation}

Therefore, $\mathcal{L}_{\sigma}$ is extended to distributions in the usual way. For $T\in\Sw'(\Hn)$, $\mathcal{L}_{\sigma}T\in\Sw'(\Hn)$ is given, for $g\in\Sw(\Hn)$, by 
\[\langle\mathcal{L}_{\sigma}T|g\rangle=\langle T|\mathcal{L}_{\sigma}g \rangle.\]
From here one can check that if $T\in\Sw'(\Hn)$,
\begin{equation}\label{multiplierDist}
\FH\mathcal{L}_{\sigma}T=(2|\lambda|)^{\sigma/2}\frac{\Gamma(|\alpha|+\frac{Q+\sigma}{4})}{\Gamma(|\alpha|+\frac{Q-\sigma}{4})}\FH T,
\end{equation}
In the $\Sw'(\Hhat)$ sense. Indeed,
\begin{align*}
\langle\FH\mathcal{L}_{\sigma}T|\eta\rangle&= \langle \mathcal{L}_{\sigma}T|\FH^{\tau}\eta\rangle\\
&=\langle T|\mathcal{L}_{\sigma}\FH^{\tau}\eta\rangle\\
&=\langle T|\FH^{\tau}[(2|\lambda|)^{\sigma/2}\frac{\Gamma(|\alpha|+\frac{Q+\sigma}{4})}{\Gamma(|\alpha|+\frac{Q-\sigma}{4})}\eta]\rangle\\
&=\langle(2|\lambda|)^{\sigma/2}\frac{\Gamma(|\alpha|+\frac{Q+\sigma}{4})}{\Gamma(|\alpha|+\frac{Q-\sigma}{4})}\FH T| \eta\rangle.
\end{align*}
From here one obtains the desired result.

\begin{proposition}\label{FHKor}
\[\FH K_{Q-\sigma}=d(n, Q-\sigma)(2|\lambda|)^{-\sigma/2}\frac{\Gamma(|\alpha|+\frac{Q-\sigma}{4})}{\Gamma(|\alpha|+\frac{Q+\sigma}{4})}\mathbbm{1}_{\{\alpha=\beta\}},\]
in the sense of distributions.
\end{proposition}
\begin{proof}
Recall that $d(\sigma,Q)^{-1}K_{Q-\sigma}$ is the fundamental solution of $\mathcal{L}_{\sigma}$. Therefore, since
\[\FH\mathcal{L}_{\sigma}K_{Q-\sigma}=(2|\lambda|)^{\sigma/2}\frac{\Gamma(|\alpha|+\frac{Q+\sigma}{4})}{\Gamma(|\alpha|+\frac{Q-\sigma}{4})}\FH K_{Q-\sigma},\]
it follows that 
\[d(n,Q-\sigma)\FH\delta_{0}=(2|\lambda|)^{\sigma/2}\frac{\Gamma(|\alpha|+\frac{Q+\sigma}{4})}{\Gamma(|\alpha|+\frac{Q-\sigma}{4})}\FH K_{Q-\sigma}.\]
Recalling the transform of $\delta_0$ we obtain
\[(2|\lambda|)^{\sigma/2}\frac{\Gamma(|\alpha|+\frac{Q+\sigma}{4})}{\Gamma(|\alpha|+\frac{Q-\sigma}{4})}\FH K_{Q-\sigma}=d(n,Q-\sigma)\mathbbm{1}_{\{\alpha=\beta\}},\]
in the sense of $\Sw'(\Hhat)$, as claimed.
\end{proof}

Hereafter, to simplify notation, I denote 
\[\mathfrak{K}_{\sigma}(\alpha,\lambda)=(2|\lambda|)^{-\sigma/2}\frac{\Gamma(|\alpha|+\frac{Q-\sigma}{4})}{\Gamma(|\alpha|+\frac{Q+\sigma}{4})}.\]
 The group Fourier transform of the Koranyi-Riez kernel, has been known for quite some time. It was originally computed by Cowling and Haagerup in \cite{CowlingHaagerup}. Their computation yields
\begin{equation}\label{groupFK}
\widehat{K}_{Q-\sigma}(\lambda)=\tilde{d}_n|\lambda|^{-\sigma/2}\frac{\Gamma(\frac{\sigma}{2})}{\Gamma(\frac{Q-\sigma)}{4})}\sum_{k=1}^{\infty}\frac{\Gamma(k+\frac{Q-\sigma}{4})}{\Gamma(k+\frac{Q+\sigma}{4})}P_k(\lambda),
\end{equation} 
where $\tilde{d}_n$ is a constant, and $P_k(\lambda)$ is the orthogonal projection of $L^2(\Rn)$ onto the eigenspace $\{\pla:|\alpha|=k\}$. 
A quick, computation renders
\[\langle P_k(\lambda)\pla,\plb\rangle_{L^2(\Rn)}=\mathbbm{1}_{\{\alpha=\beta\text{ and }|\alpha|=k\}},\]
which, at least formally, further supports Proposition \ref{FHKor}.
\subsection{Application to energy integrals}
The explicit expression for $\FH K_{\sigma}$ brings us a step closer to the coveted analogue of \eqref{FEnergy}. In fact, it is now possible to compute energies of Schwartz functions as an integral over the frequency space, $\Hhat$. For $g\in\Sw(\Hn)$, the $\sigma-$ energy of $g$ is defined as the integral
\[I_{\sigma}(g):=\int_{\Hn}\int_{\Hn}||q^{-1}p||^{-\sigma}g(q)g(p)dqdp.\]
\begin{lemma}\label{FreqEnSw}
For $g\in\Sw(\Hn)$
\[I_{\sigma}(g)=d(n,\sigma)\int_{\Ht}\mathfrak{K}_{Q-\sigma}(\alpha,\lambda)|\widehat{g}_{\Hn}(\zeta)|^2d\zeta.\]
\end{lemma}

\begin{proof}
\begin{align*}
I_{\sigma}(g)&=\int_{\Hn}g*K_{\sigma}(p)g(p)dp=\langle g*K_{\sigma}|g \rangle\\
&=\langle\FH(g*K_{\sigma})|\overline{\FH g} \rangle=\langle \FH K_{\sigma}|(\FH g)_{\tau}\cdot\overline\FH(g)\rangle\\
&=d(n,\sigma)\int_{\Ht}\mathfrak{K}_{Q-\sigma}(\alpha,\lambda)\mathbbm{1}_{\{\alpha=\beta\}}(\widehat{g}_{\Hn})_{\tau}\cdot\overline{\widehat{g}_{\Hn}}(\zeta)d\zeta\\
&=d(n,\sigma)\sum_{\alpha}\int_{\R^*}\mathfrak{K}_{Q-\sigma}(\alpha,\lambda)(\widehat{g}_{\Hn})_{\tau}\cdot\overline{\widehat{g}_{\Hn}}(\alpha,\alpha,\lambda)d\rho(\lambda)\\
&=d(n,\sigma)\sum_{\alpha}\int_{\R^*}\mathfrak{K}_{Q-\sigma}(\alpha,\lambda)\sum_{\gamma\in\N^n}\widehat{g}_{\Hn}(\alpha,\gamma,\lambda)\overline{\widehat{g}_{\Hn}(\alpha,\gamma,\lambda)}d\rho(\lambda)\\
&=d(n,\sigma)\sum_{\alpha,\gamma}\int_{\R^*}\mathfrak{K}_{Q-\sigma}(\alpha,\lambda)|\widehat{g}_{\Hn}(\alpha,\gamma,\lambda)|^2d\rho(\lambda)\\
&=d(n,\sigma)\int_{\Ht}\mathfrak{K}_{Q-\sigma}(\alpha,\lambda)|\widehat{g}_{\Hn}(\zeta)|^2d\zeta.
\end{align*}

\end{proof}
The goal is to replace $g\in\Sw(\Hn)$ by $\mu\in\Ma(\Hn)$. This can be done with the standard convolution approximation. 
\begin{proposition}\label{FreqEnH}
Let $\mu\in\Ma(\Hn)$
\[I_{\sigma}(\mu)=d(n,\sigma)\int_{\Hhat}\mathfrak{K}_{Q-\sigma}(\alpha,\lambda)|\muhatH(\zeta)|^2d\zeta\]
\end{proposition}

\begin{proof}
Let $\psi$ be a compactly supported, non-negative, smooth function with unit mass, so that defining $\psie(p)=\epsilon^{-Q}\psi(\delta_{1/\epsilon}p)$ makes $\{\psie\}_{\epsilon>0}$ a compactly supported, smooth approximation to the identity. Set $\mue=\mu*\psie$ so that $\mu\in\mathcal{C}^{\infty}_{c}(\Hn)$ as well. By Lemma \ref{FreqEnSw}, 
\[I_{\sigma}(\mue)=d(n,\sigma)\int_{\Hhat}\mathfrak{K}_{Q-\sigma}(\alpha,\lambda)|\muehat(\zeta)|^2d\zeta.\]
The idea is to show that $I_{\sigma}(\mu)=\lim_{\epsilon\to0}I_{\sigma}(\mue)$ and also \[\lim_{\epsilon\to0}\int_{\Hhat}\mathfrak{K}_{Q-\sigma}(\alpha,\lambda)|\muehat(\zeta)| ^2d\zeta=\int_{\Hhat}\mathfrak{K}_{Q-\sigma}(\alpha,\lambda)|\muhatH(\zeta)|^2d\zeta.\] I will first deal with the frequency side, and show the convergence by splitting the proof in 2 cases.

First assume that 
\[\int_{\Hhat}\mathfrak{K}_{Q-\sigma}(\alpha,\lambda)|\muhatH(\zeta)|^2d\zeta=\infty.\]
Since $\Phi$ is continuous and bounded, it follows from weak convergence that $\muehat\to\muhatH$ point-wise. Therefore, by Fatou's Lemma,
\[\infty\leq\liminf_{\epsilon\to0}\int_{\Hhat}\mathfrak{K}_{Q-\sigma}(\alpha,\lambda)|\muehat(\zeta)|^2d\zeta,\]
so the equality is trivial.\\

Now assume 
\[\int_{\Hhat}\mathfrak{K}_{Q-\sigma}(\alpha,\lambda)|\muhatH(\zeta)|^2d\zeta\leq\infty.\]
The following identity will be used: Let $H$ be a Hilbert space and $\{e_j\}_{j\in\N}$ and orthonormal basis. For an operator $A\in\mathcal{B}(H)$,
\[\sum_{j}|\langle Ae_i|e_j\rangle_{H}|^2=||Ae_i||_H^2.\]

Since $\mathfrak{K}_{Q-\sigma}$ is independent of $\beta$, it follows that $\sum_{\beta}|\muhatH(\alpha,\beta,\lambda)|^2$ is finite for each fixed $\alpha\in\N^n$ and $\varrho$-almost all $\lambda\in\R^*.$ Moreover it is integrable with respect to the product measure $\#\times\varrho$ where $\#$ stands for the counting measure on $\N^n$. Recalling \eqref{AltF}, 
\begin{align*}
\sum_{\beta}|\muehat(\alpha,\beta,\lambda)-&\muhatH(\alpha,\beta,\lambda)|^2\\
&=\sum_{\beta}|\langle(\widehat{\mu}_{\epsilon}(\lambda)-\muhat(\lambda))\pla,\plb\rangle_{L^2(\Rn)}|^2\\
&=||(\widehat{\mu}_{\epsilon}(\lambda)-\muhat(\lambda))\pla||^2_{L^2(\Rn)}.
\end{align*}
Since $\widehat{\mu}_{\epsilon}(\lambda)\to\muhat(\lambda)$ in the strong operator topology, it follows that for each fixed $\lambda\in\R^*$ and $\alpha\in\N^n$,
\[\lim_{\epsilon\to0}\sum_{\beta}|\muehat(\alpha,\beta,\lambda)-\muhatH(\alpha,\beta,\lambda)|^2=0.\]
In particular,
\[\lim_{\epsilon\to0}\sum_{\beta}|\muehat(\alpha,\beta,\lambda)|^2=\sum_{\beta}|\muhatH(\alpha,\beta,\lambda)|^2.\]
Furthermore, $\sum_{\beta}|\muehat(\alpha,\beta,\lambda)|^2$ can be dominated as follows,
\begin{align*}
\sum_{\beta}|\muehat(\alpha,\beta,\lambda)|^2&=\sum_{\beta}|\langle\widehat{\psi}_{\epsilon}(\lambda)\muhat(\lambda)\pla,\plb\rangle_{L^2(\Rn)}|^2\\
&=||\widehat{\psi}_{\epsilon}(\lambda)\muhat(\lambda)\pla||^2_{L^2(\Rn)}\\
&\leq ||\widehat{\psi}_{\epsilon}(\lambda)||_{op}^2||\muhat(\lambda)\pla||_{L^2(\Rn)}^2\\
&=||\psi_{\epsilon}||^2_{L^1(\Hn)}\sum_{\beta}|\muhatH(\alpha,\beta,\lambda)|^2=\sum_{\beta}|\muhatH(\alpha,\beta,\lambda)|^2.
\end{align*}
So, by the dominated convergence theorem,
\begin{multline*}
\lim_{\epsilon\to0}\int_{\Hhat}\mathfrak{K}_{Q-\sigma}(\alpha,\lambda)|\muehat(\zeta)|^2d\zeta\\=\sum_{\alpha\in\N^n}\int_{\R^*}\mathfrak{K}_{Q-\sigma}(\alpha,\lambda)\lim_{\epsilon\to0}\sum_{\beta}|\muehat(\zeta)|^2d\zeta=\int_{\Hhat}\mathfrak{K}_{Q-\sigma}(\alpha,\lambda)|\muhatH(\zeta)|^2d\zeta.
\end{multline*}

Now the energy on the spatial domain is studied.

Once again, if $I_{\sigma}(\mu)=\infty$, Fatou's Lemma gives,
\[\infty\leq\liminf_{\epsilon\to0}I_{\sigma}(\mue),\]
so the equality is trivial.\\

Assume $I_{\sigma}(\mu)<\infty$. By expanding the convolution $\psi_{\epsilon}\ast\mu$, and applying Fubini's theorem,
\[I_{\sigma}(\mue)=\iint\left(\iint ||q^{-1}p||_{\Hn}^{-\sigma}\psi_{\epsilon}(a^{-1}p)\psi_{\epsilon}(b^{-1}q)dpdq\right)d\mu(a)d\mu(b).\]
Using $W=\delta_{1/\epsilon}(a^{-1}p)$ and $Z=\delta_{1/\epsilon}(b^{-1}q)$, the inner integral becomes
\[\iint ||\delta_{\epsilon}(Z)b^{-1}a\delta_{\epsilon}(W)||_{\Hn}^{-\sigma}\psi(Z)\psi(W)dZdW.\]

As $\epsilon\to0$, this integral goes to $||b^{-1}a||^{-\sigma}_{\Hn}$ as long as $a\neq b$. Moreover, it shows that 
\[\iint ||q^{-1}p||_{\Hn}^{-\sigma}\psi_{\epsilon}(a^{-1}p)\psi_{\epsilon}(b^{-1}q)dpdq\lesssim ||b^{-1}a||^{-\sigma}_{\Hn}.\]
By assumption, $||b^{-1}a||^{-\sigma}_{\Hn}$ is integrable, so by the Dominated Convergence Theorem
\begin{multline*}
\lim_{\epsilon\to0}I_{\sigma}(\mue)=\lim_{\epsilon\to0}\iint\left(\iint ||q^{-1}p||_{\Hn}^{-\sigma}\psi_{\epsilon}(a^{-1}p)\psi_{\epsilon}(b^{-1}q)dpdq\right)d\mu(a)d\mu(b)\\=\iint\left(\lim_{\epsilon\to0}\iint ||q^{-1}p||_{\Hn}^{-\sigma}\psi_{\epsilon}(a^{-1}p)\psi_{\epsilon}(b^{-1}q)dpdq\right)d\mu(a)d\mu(b)\\=\iint ||a^{-1}b||_{\Hn}^{-\sigma}d\mu(a)d\mu(b)=I_{\sigma}(\mu).
\end{multline*}
This completes the proof.
\end{proof}


\section{Acknowledgements}
The author would like to thank Professor Jeremy Tyson for his advising and guidance during this project. The author would also like to thank Professor Marius Junge for many helpful discussions and insights related to this project.

\bibliographystyle{plain}

\bibliography{Biblio2}

\begin{thebibliography}{10}

\bibitem{BCD1}
Hajer {Bahouri}, Jean-Yves {Chemin}, and Raphael {Danchin}.
\newblock {A Frequency Space for the Heisenberg Group}.
\newblock {\em arXiv e-prints}, page arXiv:1609.03850, Sep 2016.

\bibitem{BCD2}
Hajer {Bahouri}, Jean-Yves {Chemin}, and Raphael {Danchin}.
\newblock {Tempered distributions and Fourier transform on the Heisenberg
  group}.
\newblock {\em arXiv e-prints}, page arXiv:1705.02195, May 2017.

\bibitem{BDR}
Chal Benson, A.~H. Dooley, and Gail Ratcliff.
\newblock Fundamental solutions for powers of the {H}eisenberg sub-{L}aplacian.
\newblock {\em Illinois J. Math.}, 37(3):455--476, 1993.

\bibitem{BOO}
Thomas Branson, Gestur \'{O}lafsson, and Bent \O~rsted.
\newblock Spectrum generating operators and intertwining operators for
  representations induced from a maximal parabolic subgroup.
\newblock {\em J. Funct. Anal.}, 135(1):163--205, 1996.

\bibitem{BFM}
Thomas~P. Branson, Luigi Fontana, and Carlo Morpurgo.
\newblock Moser-{T}rudinger and {B}eckner-{O}nofri's inequalities on the {CR}
  sphere.
\newblock {\em Ann. of Math. (2)}, 177(1):1--52, 2013.

\bibitem{Brislawn}
Chris Brislawn.
\newblock Kernels of trace class operators.
\newblock {\em Proc. Amer. Math. Soc.}, 104(4):1181--1190, 1988.

\bibitem{CDPT}
Luca Capogna, Donatella Danielli, Scott~D. Pauls, and Jeremy~T. Tyson.
\newblock {\em An introduction to the {H}eisenberg group and the
  sub-{R}iemannian isoperimetric problem}, volume 259 of {\em Progress in
  Mathematics}.
\newblock Birkh\"auser Verlag, Basel, 2007.

\bibitem{CowlingHaagerup}
Michael Cowling and Uffe Haagerup.
\newblock Completely bounded multipliers of the {F}ourier algebra of a simple
  {L}ie group of real rank one.
\newblock {\em Invent. Math.}, 96(3):507--549, 1989.

\bibitem{Falconer4}
K.~J. Falconer.
\newblock Classes of sets with large intersection.
\newblock {\em Mathematika}, 32(2):191--205 (1986), 1985.

\bibitem{Erdogan}
K.~J. Falconer.
\newblock On the {H}ausdorff dimensions of distance sets.
\newblock {\em Mathematika}, 32(2):206--212 (1986), 1985.

\bibitem{Folland3}
G.~B. Folland.
\newblock A fundamental solution for a subelliptic operator.
\newblock {\em Bull. Amer. Math. Soc.}, 79:373--376, 1973.

\bibitem{Folland2}
Gerald~B. Folland.
\newblock {\em Harmonic analysis in phase space}, volume 122 of {\em Annals of
  Mathematics Studies}.
\newblock Princeton University Press, Princeton, NJ, 1989.

\bibitem{FrankGonzMontiTan}
Rupert~L. Frank, Mar\'{\i}a del~Mar Gonz\'{a}lez, Dario~D. Monticelli, and
  Jinggang Tan.
\newblock An extension problem for the {CR} fractional {L}aplacian.
\newblock {\em Adv. Math.}, 270:97--137, 2015.

\bibitem{FrankLieb}
Rupert~L. Frank and Elliott~H. Lieb.
\newblock Sharp constants in several inequalities on the {H}eisenberg group.
\newblock {\em Ann. of Math. (2)}, 176(1):349--381, 2012.

\bibitem{Geller}
Daryl Geller.
\newblock Fourier analysis on the {H}eisenberg group. {I}. {S}chwartz space.
\newblock {\em J. Functional Analysis}, 36(2):205--254, 1980.

\bibitem{Heyer}
Herbert Heyer.
\newblock L'analyse de {F}ourier non-commutative et applications \`a la
  th\'{e}orie des probabilit\'{e}s.
\newblock {\em Ann. Inst. H. Poincar\'{e} Sect. B (N.S.)}, 4:143--164, 1968.

\bibitem{JohaWallch}
Kenneth~D. Johnson and Nolan~R. Wallach.
\newblock Composition series and intertwining operators for the spherical
  principal series. {I}.
\newblock {\em Trans. Amer. Math. Soc.}, 229:137--173, 1977.

\bibitem{Kaufman1}
Robert Kaufman.
\newblock On {H}ausdorff dimension of projections.
\newblock {\em Mathematika}, 15:153--155, 1968.

\bibitem{Marstrand}
J.~M. Marstrand.
\newblock Some fundamental geometrical properties of plane sets of fractional
  dimensions.
\newblock {\em Proc. London Math. Soc. (3)}, 4:257--302, 1954.

\bibitem{Mattila3}
Pertti Mattila.
\newblock Hausdorff dimension, orthogonal projections and intersections with
  planes.
\newblock {\em Ann. Acad. Sci. Fenn. Ser. A I Math.}, 1(2):227--244, 1975.

\bibitem{Mattila6}
Pertti Mattila.
\newblock On the {H}ausdorff dimension and capacities of intersections.
\newblock {\em Mathematika}, 32(2):213--217 (1986), 1985.

\bibitem{Mattila1}
Pertti Mattila.
\newblock {\em Geometry of sets and measures in {E}uclidean spaces}, volume~44
  of {\em Cambridge Studies in Advanced Mathematics}.
\newblock Cambridge University Press, Cambridge, 1995.
\newblock Fractals and rectifiability.

\bibitem{RoncalThang}
Luz Roncal and Sundaram Thangavelu.
\newblock Hardy's inequality for fractional powers of the sublaplacian on the
  {H}eisenberg group.
\newblock {\em Adv. Math.}, 302:106--158, 2016.

\bibitem{Siebert}
Eberhard Siebert.
\newblock Fourier analysis and limit theorems for convolution semigroups on a
  locally compact group.
\newblock {\em Adv. in Math.}, 39(2):111--154, 1981.

\bibitem{Stone}
Marshall~Harvey Stone.
\newblock {\em Linear transformations in {H}ilbert space}, volume~15 of {\em
  American Mathematical Society Colloquium Publications}.
\newblock American Mathematical Society, Providence, RI, 1990.
\newblock Reprint of the 1932 original.

\bibitem{Tangavelu1}
Sundaram Thangavelu.
\newblock {\em Harmonic analysis on the {H}eisenberg group}, volume 159 of {\em
  Progress in Mathematics}.
\newblock Birkh\"auser Boston, Inc., Boston, MA, 1998.

\bibitem{Tangavelu2}
Sundaram Thangavelu.
\newblock {\em An introduction to the uncertainty principle}, volume 217 of
  {\em Progress in Mathematics}.
\newblock Birkh\"{a}user Boston, Inc., Boston, MA, 2004.
\newblock Hardy's theorem on Lie groups, With a foreword by Gerald B. Folland.

\bibitem{VonN}
J.~v.~Neumann.
\newblock Die {E}indeutigkeit der {S}chr\"{o}dingerschen {O}peratoren.
\newblock {\em Math. Ann.}, 104(1):570--578, 1931.

\bibitem{WangWu}
Hai-meng Wang and Qing-yan Wu.
\newblock On fundamental solution for powers of the sub-{L}aplacian on the
  {H}eisenberg group.
\newblock {\em Appl. Math. J. Chinese Univ. Ser. B}, 32(3):365--378, 2017.

\end{thebibliography}

\end{document}